\documentclass[english]{article}
\usepackage[T1]{fontenc}
\usepackage{caption}

\usepackage[latin9]{inputenc}
\usepackage{amsmath,amsthm} 
\usepackage{graphicx}
\usepackage[dvipsnames,svgnames,x11names,hyperref]{xcolor}
\usepackage{float}
\usepackage{amssymb}
\usepackage{enumitem}
\usepackage{natbib}

\setcitestyle{authoryear, open={(},close={)}}
\usepackage{bussproofs}
\usepackage[letterpaper]{geometry} 
\usepackage{thmtools,thm-restate}
\usepackage{tikz}
\usepackage{tikz-cd}
\usepackage{pgf} 
\usepackage{mathtools}

\usetikzlibrary{patterns,automata,arrows,shapes,snakes,topaths,trees,backgrounds,positioning,through,calc}
\usepackage{setspace}
\usepackage{multicol}
\usepackage{graphicx}
\usepackage{stmaryrd}  
\usepackage{comment}
\usepackage{tablefootnote}
\usepackage{thmtools,thm-restate}

\usepackage{thmtools}
\usepackage{thm-restate}

\usepackage{hyperref}

\usepackage{cleveref}
\usepackage{tikz}
\usetikzlibrary{arrows,chains,matrix,positioning,scopes}
\makeatletter
\tikzset{join/.code=\tikzset{after node path={%
			\ifx\tikzchainprevious\pgfutil@empty\else(\tikzchainprevious)%
			edge[every join]#1(\tikzchaincurrent)\fi}}}
\makeatother
\tikzset{>=stealth',every on chain/.append style={join},
	every join/.style={->}}
\tikzstyle{labeled}=[execute at begin node=$\scriptstyle,
execute at end node=$]

\usepackage{hyperref}

\hypersetup{colorlinks=true, citecolor=ForestGreen, linkcolor=ForestGreen}  

\usepackage{mathrsfs}

\newlist{myitemize}{itemize}{8}

\usepackage{booktabs}

\theoremstyle{definition}
\newtheorem{theorem}{Theorem}[section]

\newtheorem{proposition}[theorem]{Proposition}
\newtheorem{example}[theorem]{Example}
\newtheorem{definition}[theorem]{Definition}

\newtheorem{remark}[theorem]{Remark} 
 
\newtheorem{corollary}[theorem]{Corollary} 
 
\usepackage[page,toc,titletoc,title]{appendix}

\numberwithin{figure}{section}


\onehalfspace

\begin{document} 
	
	\title{On Accuracy and Coherence with\\ Infinite Opinion Sets\footnote{	Thanks to participants of the Berkeley-Stanford Logic Circle (April 2019), Probability and Logic Conference (July 2019), Berkeley Formal Epistemology Reading Course (October 2019), and Stanford Logic and Formal Philosophy Seminar (November 2019), to whom earlier versions of this paper were presented. For helpful comments and discussion, thanks to Kenny Easwaran, Craig Evans, Wesley Holliday, Thomas Icard, Kiran Luecke, Calum McNamara, Sven Neth, Richard Pettigrew, Eric Raidl, Teddy Seidenfeld, and James Walsh. Special thanks to two anonymous referees at \textit{Philosophy of Science} for their very useful feedback.}}
	\author{Mikayla Kelley \\{\small Stanford University}}
 \date{{\small Penultimate draft. Forthcoming in \textit{Philosophy of Science}.}}
	\maketitle
	
	\begin{abstract}
There is a well-known equivalence between avoiding accuracy dominance and having probabilistically coherent credences (see, e.g., \citealt{definetti}, \citealt{Joyce}, \citealt{predd}, \citealt{1schervish2009}, \citealt{pettigrew}). However, this equivalence has been established only when the set of propositions on which credence functions are defined is finite. In this paper, I establish connections between accuracy dominance and coherence when credence functions are defined on an infinite set of propositions. In particular, I establish the necessary results to extend the classic accuracy argument for probabilism originally due to \citet{Joyce1998} to certain classes of infinite sets of propositions including countably infinite partitions.
\end{abstract}

	\section{Introduction}
	A central norm in the epistemology of partial belief is probabilism: a person's degrees of belief---or \textit{credences}---should satisfy the laws of probability.\footnote{This paper is based on work done in \citealt{kelleym}.} There is a long tradition in the spirit of \cite{Savage} and \cite{definetti} of appealing to the epistemic virtue of accuracy to justify probabilism (also see \citealt{Rosenkrantz1981-ROSFAA-3}). One particular form of argument is the accuracy dominance argument for probabilism introduced by \cite{Joyce1998}. Let a set $\mathcal{F}$ of propositions be an \textit{opinion set} and a function $c:\mathcal{F}\to [0,1]$ a \textit{credence function on} $\mathcal{F}$. Let a credence function be \textit{coherent} if it satisfies the axioms of probability.  A credence function $c'$ on $\mathcal{F}$ \textit{accuracy dominates} a credence function $c$ on $\mathcal{F}$ if $c$ is more inaccurate than $c'$ no matter how the world turns out to be (where inaccuracy is precisified as in Section \ref{finitecase}). Then the existing accuracy dominance arguments purport to vindicate probabilism by showing that a credence function is not accuracy dominated if and only if it is coherent.
	
	% 	\footnote{Add citations to Staffel's Unsettled Thoughts and Pettgrew's paper without additivity}

	However, there is a limitation to almost all of the literature on accuracy dominance arguments for probabilism: the opinion set is assumed to be finite.\footnote{In an unpublished manuscript, \cite{walsh} proves an accuracy dominance result in the countably infinite context, to which we return in Section \ref{countablecase1}. In a related but distinct area, \cite{hutteger} and \cite{easwaran} extend to the infinite setting part of the literature on using minimization of expected inaccuracy to vindicate epistemic principles. See, e.g., \citealt{greaves}. \cite{schervishreference} prove that in certain countably infinite cases, coherence is sufficient to avoid \textit{strong dominance}. \cite{1schervish2009} and \cite{STEEGER201934} explore a different way to weaken the assumption that the opinion set is finite. We return to their work in Section \ref{countablecase2}.} Indeed, \cite{definetti}, \cite{lindley}, \cite{Joyce1998, Joyce}, \cite{predd}, \cite{leitgebi,leitgebii}, and \cite{pettigrew} all  establish their dominance results only for finite opinion sets.\footnote{The same holds for accuracy dominance results pertaining to approximating coherence (\citealt{debona}, \citealt[ch.~5]{Staffel2019-STAUTA-2}) and accuracy dominance results that significantly weaken the additivity assumption on the measure of inaccuracy (\citealt{pettigrew21}, \citealt{nielsen}), though neither will be my focus here.} In this paper, I remove this assumption and prove dominance results that I hope to be useful in evaluating the extent to which accuracy dominance arguments for probabilism succeed when the opinion set is infinite.  
	
	I begin in Section \ref{finitecase} by reviewing the mathematical framework and the standard dominance result for finite opinion sets. Sections \ref{countablecase1}-\ref{uncountablecase} are concerned with accuracy and coherence in the infinite setting. In Sections \ref{countablecase1} and \ref{countablecase2}, I make headway on characterizing the opinion sets and accuracy measures for which there is an equivalence between coherence and avoiding dominance as in the finite case. In Section \ref{uncountablecase}, I extend the accuracy framework to the uncountable setting and prove that coherence is necessary to avoid dominance on uncountable opinion~sets. I conclude in Section \ref{discussion} with a discussion of the results established in  Sections \ref{countablecase1}-\ref{uncountablecase}.

	\section{The Finite Case}\label{finitecase}
	We first set up the framework that will be used throughout the paper. Fix a set $W$ (not necessarily finite) which represents the set of \textit{possible worlds} and, for now, a finite set $\mathcal{F}\subseteq \mathcal{P}(W)$ of \textit{propositions} that represents an \textit{opinion set}---the set of propositions that an agent has beliefs about.
	\begin{definition}\label{Algebra} An \textit{algebra} over $W$ is a subset $\mathcal{F}^*\subseteq \mathcal{P}(W)$ such that: 
		\begin{enumerate}
			\item $W\in \mathcal{F}^*$;
			\item if $p,p'\in \mathcal{F}^*$, then $p\cup p'\in \mathcal{F}^*$;
			\item if $p\in \mathcal{F}^*$, then $W\setminus p\in \mathcal{F}^*$.
		\end{enumerate}
	\end{definition}
	
	\begin{definition}\label{credence}
		\begin{enumerate}
			\item[i.] A \textit{credence function} on an opinion set $\mathcal{F}$ is a function from $\mathcal{F}$ to $[0,1]$. 
			\item[ii.] A credence function $c$ is \textit{coherent} if it can be extended to a finitely additive probability function on an algebra $\mathcal{F}^*$ over $W$ containing $\mathcal{F}$. That is, there is an algebra $\mathcal{F}^*\supseteq \mathcal{F}$ over $W$ and a function $c^*:\mathcal{F}^*\to [0,1]$ such that:
			\begin{enumerate}
				\item $c^*(p)=c(p)$ for all $p\in \mathcal{F}$;
				\item $c^*(p\cup p')=c^*(p)+c^*(p')$ for $p,p'\in \mathcal{F}^*$ with $p\cap p'=\varnothing$;
				\item $c^*(W)=1$.
			\end{enumerate}
			\item[iii.] A credence function that is not coherent is \textit{incoherent}. 
			
		\end{enumerate}
	\end{definition}

	\begin{remark}\label{tupelrep} If $\mathcal{F}=\{p_1,\ldots,p_n\}$, a credence function $c$ over $\mathcal{F}$ can be identified with the vector $(c(p_1),\ldots,c(p_n))\in [0,1]^n$. Thus the space of all credence functions over $\mathcal{F}$ can be identified with $[0,1]^n\subseteq\mathbb{R}^n$. We often simplify notation by setting $c_i:=c(p_i)$.
	\end{remark}
	
	We now introduce an important subclass of the class of all credence functions, namely the (coherent) credence functions that match the truth values of $\mathcal{F}$ at a world $w$ exactly.  
	
	\begin{definition}\label{omniscient} Fix an opinion set $\mathcal{F}$. For each $w\in W$, let $v_{w}:\mathcal{F}\to \{0,1\}$ be defined by $v_w(p)=1$ if and only if $w\in p$. We call $v_w$ the \textit{omniscient credence function at world $w$}. We let $\mathcal{V}_{\mathcal{F}}$ denote the set of all omniscient credence functions on $\mathcal{F}$. Note that $|\mathcal{V}_{\mathcal{F}}|\leq 2^{|\mathcal{F}|}$.
	\end{definition}

	Next, we specify the inaccuracy measures that we will be concerned with in this section.  Fix a finite opinion set $\mathcal{F}$, and let $\mathcal{C}$ denote the set of credence functions on $\mathcal{F}$. We define an \textit{inaccuracy measure} to be a function of the form 
	\[\mathscr{I}:\mathcal{C}\times W\to [0,\infty].\]
	The class of inaccuracy measures we consider is a generalization of the class defended by \cite{pettigrew}: the inaccuracy measures defined in terms of what we call a \textit{quasi-additive Bregman divergence}.  It is a subclass of the inaccuracy measures assumed in \citealt{predd}.\footnote{Using terminology from Definition \ref{divergence}, \citeauthor{predd} consider a more general class in allowing different one-dimensional Bregman divergences for different propositions.}
	
	\begin{definition}\label{divergence} Suppose $\mathfrak{D}:[0,1]^n\times [0,1]^n\to [0,\infty]$.
		\begin{enumerate}
			\item\label{divergence1} $\mathfrak{D}$ is a \textit{divergence} if $\mathfrak{D}(\mathbf{x},\mathbf{y})\geq 0$ for all $\mathbf{x},\mathbf{y}\in [0,1]^n$ with equality if and only if $\mathbf{x}=\mathbf{y}$.
			\item\label{divergence2} $\mathfrak{D}$ is \textit{quasi-additive} if there exists a function $\mathfrak{d}:[0,1]^2\to [0,\infty]$ and a sequence of elements $\{a_i\}_{i=1}^n$ from $(0,\infty)$ such that 
			\[\mathfrak{D}(\mathbf{x},\mathbf{y})=\sum_{i=1}^na_i\mathfrak{d}(x_i,y_i),\]
			in which case we say $\mathfrak{D}$ is \textit{generated} by $\mathfrak{d}$ and $\{a_i\}_{i=1}^n$.
			\item\label{divergence3} $\mathfrak{D}$ is a \textit{quasi-additive Bregman divergence} if $\mathfrak{D}$ is a quasi-additive divergence generated by $\mathfrak{d}$ and $\{a_i\}_{i=1}^n$, and in addition there is a function $\varphi:[0,1]\to \mathbb{R}$ such that:
			\begin{enumerate}
				\item $\varphi$ is continuous and strictly convex on $[0,1]$;\footnote{Therefore $\varphi$ is bounded, as it is a continuous function on a compact interval.}
				\item $\varphi$ is continuously differentiable on $(0,1)$ with the formal definition
				\[\varphi'(i):=\lim_{x\to i}\varphi'(x)\]
				for $i\in \{0,1\}$;\footnote{We do not require $\varphi'(i)<\infty$ for $i\in \{0,1\}$.}
				\item for all $x,y\in[0,1]$, we have 
				\[\mathfrak{d}(x,y)=\varphi(x)-\varphi(y)-\varphi'(y)(x-y).\]
				We call such a $\mathfrak{d}$ a \textit{one-dimensional Bregman divergence}.
			\end{enumerate}
		\end{enumerate}
	\end{definition}
	\noindent We take the inaccuracy of a credence function $c$ at a world $w$ to be the distance between $c$ and the omniscient credence function $v_w$, where distance is measured with a quasi-additive Bregman divergence.
	
	    \begin{definition}
		Let a \textit{legitimate inaccuracy measure} be an inaccuracy measure given by \[\mathscr{I}(c,w)=\mathfrak{D}(v_w,c),\] where $\mathfrak{D}$ is a quasi-additive Bregman divergence. 
	\end{definition}
	
	By allowing different weights depending on the proposition, we can accommodate the intuition that some propositions are more important to know than others.\footnote{Though see \citealt{Levinstein2019-LEVAOO-3} for an argument that one should expect weights to vary with respect to worlds as well as propositions.} Even if one thinks that inaccuracy measures should be additive, as \cite{pettigrew} does, relaxing this restriction makes our results more widely relevant.  A popular example of an additive legitimate inaccuracy measure is the Brier score (see Section 12, ``Homage to the Brier Score,'' of \citealt{Joyce}):
	\[\mathscr{I}(c,w)=\sum_{i=1}^n(v_w(p_i)-c(p_i))^2.\]
	\begin{remark}\label{strictlyproperremark}
	The class of additive Bregman divergences is the class of additive and continuous \textit{strictly proper scoring rules}. See \citealt[p.~66]{pettigrew}. Also see, e.g., \citealt{bregmandivergences2} and \citealt{bregmandivergences1} for more details on Bregman divergences as well as their connection to strictly proper scoring rules. 
	\end{remark}

	We now recall the dominance result connecting coherence to accuracy dominance when the opinion set is finite. It was first proved for the Brier score by \citeauthor{definetti} (\citeyear[pp.~87-90]{definetti}) and extended to any legitimate inaccuracy measure by \cite{predd}.\footnote{See Section 7 of \citealt{predd}.} See \citealt{1schervish2009} and \citealt{pettigrew21} (see also \citealt{nielsen}) for further generalizations of the finite result.
	
	\begin{definition}\label{dominance}For each pair of credence functions $c, c^*$ over $\mathcal{F}$:
		\begin{enumerate}
			\item \label{dominanceweak}$c^*$ \textit{weakly dominates} $c$ relative to an inaccuracy measure $\mathscr{I}$ if $\mathscr{I}(c,w)\geq \mathscr{I}(c^*,w)$ for all $w\in W$ and $\mathscr{I}(c,w)>\mathscr{I}(c^*,w)$ for some $w\in W$;
			\item\label{dominancestrong} $c^*$ \textit{strongly dominates} $c$ relative to $\mathscr{I}$ if $\mathscr{I}(c,w)>\mathscr{I}(c^*,w)$ for all $w\in W$.
		\end{enumerate}
	\end{definition}
	\begin{theorem}[\citealt{definetti}, \citealt{predd}]\label{step3result} Let $\mathcal{F}$ be a finite opinion set, $\mathscr{I}$ a legitimate inaccuracy measure, and $c$ a credence function on $\mathcal{F}$. Then the following are equivalent:
		\begin{enumerate}
			\item $c$ is not strongly dominated;
			\item $c$ is not weakly dominated;
			\item $c$ is coherent.
		\end{enumerate}
	Further, if $c$ is incoherent, then $c$ is strongly dominated by a coherent credence function.
	\end{theorem}

	On the basis of Theorem \ref{step3result}, authors in the accuracy literature conclude that an incoherent credence function is objectionable because there is an undominated coherent credence function that does strictly better in terms of accuracy, no matter how the world turns out to be, whereas coherent credence functions are not accuracy dominated in this way. Since it is the basis of the accuracy dominance argument for probabilism in the finite case, Theorem \ref{step3result} is the result we would like to extend to infinite opinion sets. We now make progress toward this goal when $\mathcal{F}$ is countably infinite.

	\textit{Note that any missing proofs in Sections \ref{countablecase1}-\ref{uncountablecase} can be found in the Appendix. } 
	\section{The Countably Infinite Case: Coherence is Necessary}\label{countablecase1}
	\subsection{Generalized Legitimate Inaccuracy Measures}
	We begin with a discussion of how to measure inaccuracy in the countably infinite setting. Fix a countably infinite opinion set $\mathcal{F}$ over a set $W$ of worlds (of arbitrary cardinality). Let $\mathcal{C}$ be the set of credence functions over $\mathcal{F}$, which can be identified with $[0,1]^\infty$ (see Remark \ref{tupelrep}). An \textit{inaccuracy measure} remains a map from $\mathcal{C}\times W$ to $[0,\infty]$.
	
	The class of inaccuracy measures that we use are defined in terms of generalizations of quasi-additive Bregman divergences.
	
	\begin{definition}\label{quasiadditivebregman}
		Suppose $\mathfrak{D}:[0,1]^\infty\times [0,1]^\infty\to[0,\infty]$. Then we call $\mathfrak{D}$ a \textit{generalized quasi-additive Bregman divergence} if \[\mathfrak{D}(\mathbf{x},\mathbf{y})=\sum_{i=1}^\infty a_i \mathfrak{d}(x_i,y_i),\]
		where $\mathfrak{d}$ is a bounded\footnote{Boundedness is assumed for technical reasons.} one-dimensional Bregman divergence as in Definition \ref{divergence}.\ref{divergence3} and $\{a_i\}_{i=1}^\infty$ a sequence of elements from $ (0,\infty)$ with $\sup_{i}a_i<\infty$.\footnote{Recall that $\sup_ia_i=a\in \mathbb{R}\cup\{+\infty, -\infty\}$ such that $a_i\leq a$ for all $i\in \mathbb{N}$ and for any $b<a$, there is some $a_i$ such that $b<a_i\leq a$.} 
	\end{definition}
	\begin{remark}\label{reductiononphi}
		Note that $\mathfrak{d}$---defined in terms of $\varphi$---being bounded is equivalent to $\varphi'$ being bounded on $[0,1]$. Further, we may assume that $\varphi(0)=\varphi'(0)=0$ since $\mathfrak{d}_\varphi=\mathfrak{d}_{\bar{\varphi}}$ if $\varphi$ and $\bar{\varphi}$ differ by a linear function.\footnote{Proof: Let $\bar{\varphi}(x)=\varphi(x)+ax+b$. Then $\mathfrak{d}_{\bar{\varphi}}(x,y)=\varphi(x)+ax+b-\varphi(y)-ay-b-(\varphi'(y)+a)(x-y)=\varphi(x)+ax+b-\varphi(y)-ay-b-\varphi'(y)(x-y)-ax+ay=\varphi(x)-\varphi(y)-\varphi'(y)(x-y)=\mathfrak{d}_{\varphi}(x,y)$. Further, if $\varphi$ satisfies the conditions in Definition \ref{divergence}.\ref{divergence3}, then $\bar{\varphi}$ does as well.}
	\end{remark}
	\noindent In the appendix, we show that generalized quasi-additive Bregman divergences are examples of what \cite{csiszar} calls \textit{Bregman distances}, which are generalizations of quasi-additive Bregman divergences defined on spaces of non-negative functions.
	
	Suggestively, we make the following definition.
	\begin{definition} Given an enumeration of $\mathcal{F}$,\footnote{The choice of enumeration does not matter since the terms in the infinite sum defining inaccuracy are non-negative. Thus convergence is absolute and independent of order.} let a \textit{generalized legitimate inaccuracy measure} be an inaccuracy measure $\mathscr{I}:\mathcal{C}\times W\to [0,\infty]$ given by 
		\begin{equation}\label{specificinaccuracy}
		\mathscr{I}(c,w)=\mathfrak{D}(v_w,c)
		\end{equation}
		for $\mathfrak{D}$ a generalized quasi-additive Bregman divergence.\footnote{Note that in the infinite setting, we may have $v_{w_1}=v_{w_2}$ with $w_1\neq w_2$.}
	\end{definition}
	\noindent Notice that the Brier score extends to a generalized legitimate inaccuracy measure, namely the squared $\ell^2(\mathcal{F})$ norm
	\begin{equation}\label{generalizedbrier}\mathscr{I}(c,w)=||v_w-c||_{\ell^2(\mathcal{F})}^2=\sum_{i=1}^\infty (v_w(p_i)-c(p_i))^2.
	\end{equation}
	We call (\ref{generalizedbrier}) the \textit{generalized Brier score}.
	
	The name ``generalized legitimate inaccuracy measure'' is motivated by the observation that a generalized legitimate inaccuracy measure naturally restricted to the finite opinion sets is a legitimate inaccuracy measure. This is because 1) for both the generalized and finite legitimate inaccuracy measures, the score of an individual proposition is defined by a one-dimensional Bregman divergence, and 2) for both the generalized and finite legitimate inaccuracy measures, the scores of individual propositions are combined in a weighted additive way to give a score for the entire credence function. To use the terminology of \cite{leitgebi}, in the finite and countably infinite setting, the local scores are the same and the global scores relate to the local scores in the same way. These observations support the view that, insofar as quasi-additive Bregman divergences are the appropriate functions to use for measuring inaccuracy in the finite setting, generalized quasi-additive Bregman divergences are the appropriate functions to use for measuring inaccuracy in the countably infinite setting (see Section \ref{discussion} for further discussion of generalized legitimate inaccuracy measures).\footnote{\label{axiom}The class of generalized legitimate inaccuracy measures could also be justified by defending the following axiom, which picks out the generalized legitimate inaccuracy measures: let $\mathcal{F}$ be a countable opinion set and let $\mathcal{F}_1\subseteq \mathcal{F}_2\subseteq \ldots$ be a sequence of increasing subsets of $\mathcal{F}$ whose union is $\mathcal{F}$. Then there are legitimate inaccuracy measures $\mathscr{I}_1,\mathscr{I}_2,\ldots$ generated by the same one-dimensional Bregman divergence and compatible weights such that the inaccuracy of a credence function $c$ on $\mathcal{F}$ at world $w$ is given by $\mathscr{I}(c,w)=\lim_{n\to \infty}\mathscr{I}_n(c|_n,w)$, where $c|_n$ is the restriction of $c$ to $\mathcal{F}_n$. By ``compatible weights'', we mean that if $\mathscr{I}_n$ and $\mathscr{I}_{n+1}$ are generated by weights $\{a_i\}_{i=1}^{|\mathcal{F}_n|}$ and $\{b_i\}_{i=1}^{|\mathcal{F}_{n+1}|}$, respectively, then $a_i=b_i$ for $1\leq i\leq |\mathcal{F}_n|$.}
	\subsection{Coherence is Necessary}
	We now state one of our main results: coherence is necessary to avoid accuracy dominance in the countably infinite case. See the Appendix for the proof.
	
	\begin{restatable}{theorem}{inftheorem}\label{inftheorem1}
		Let $\mathcal{F}$ be a countably infinite opinion set, $\mathscr{I}$ a generalized legitimate inaccuracy measure, and $c$ an incoherent credence function. Then:
		\begin{enumerate}
			\item \label{inftheorem11} $c$ is weakly dominated relative to $\mathscr{I}$ by a coherent credence function; and
			\item \label{inftheorem12} if $\mathscr{I}(c,w)<\infty$ for each $w\in W$, then $c$ is strongly dominated relative to $\mathscr{I}$ by a coherent credence function.
			\end{enumerate}
		\end{restatable}
% \begin{remark}\label{unbounded}
% Theorem \ref{inftheorem1} in fact holds for a larger class of inaccuracy measures: one may drop the assumption that $\mathfrak{d}$ is bounded in Definition \ref{quasiadditivebregman}.\footnote{I thank an anonymous referee for prompting this remark.}
% \end{remark}
	\begin{remark}\label{finitewherefinite}
	By analyzing the proof of Theorem \ref{inftheorem1}, one can see that the most general way to state the theorem is: assume $c$ is incoherent; if $\mathscr{I}(c,w)<\infty$ for some $w$, then there is a coherent credence function $d$ such that $\mathscr{I}(d,w)<\mathscr{I}(c,w)$ for all $w$ such that $\mathscr{I}(c,w)<\infty$; if $\mathscr{I}(c,w)=\infty$ for all $w\in W$, then any omniscient credence function weakly dominates $c$.
	\end{remark}
	\begin{remark}\label{dominatedbycoherent}
	The following is easy to prove from the results of \cite{1schervish2009}: any incoherent credence function $c$ over a countably infinite opinion set is weakly dominated but not necessarily by a \textit{coherent} credence function; and  if $\mathscr{I}(c,w)<\infty$ for each $w\in W$, then $c$ is strongly dominated but not necessarily by a \textit{coherent} credence function.\footnote{Proof sketch: If $c$ is incoherent, then there is some finite $\overline{\mathcal{F}}\subseteq \mathcal{F}$ on which $c$ is incoherent. Restrict $c$ to $c|_{\overline{\mathcal{F}}}$ on $\overline{\mathcal{F}}$. Then by Theorem \ref{step3result}, there is some $\overline{d}$ that strongly dominates $c|_{\overline{\mathcal{F}}}$. Extend $\overline{d}$ to a credence function $d$ on $\mathcal{F}$ by copying $c$ off of $\overline{\mathcal{F}}$. Then so long as $c$ has finite inaccuracy at some world, $d$ will weakly dominate $c$.} Thus the value in the proof strategy to come is that the dominating credence function is proven to be coherent, which is analogous to the finite case.\footnote{Thanks to Teddy Seidenfeld for suggesting this connection to the finite case.}\textsuperscript{,}\footnote{\label{dominancefootnote}Further, it is often argued that not all dominated credence functions are irrational---only those that are dominated by a credence function which is itself not dominated (see discussion of various dominance principles in \citealt[p.~22]{pettigrew}). For the opinion sets and inaccuracy measures dicussed in Section \ref{countablecase2}, the undominated credence functions will be precisely the coherent credence functions, and so the added strength of Theorem \ref{inftheorem1} is normatively important, as well. }
	\end{remark}
	
	We note that one direction of \citeauthor{walsh}'s (\citeyear{walsh}) accuracy dominance result follows immediately from Theorem \ref{inftheorem1}. We first recall his result.
	\begin{theorem}[\citealt{walsh}]\label{walshtheorem} Let $\mathcal{F}$ be a countably infinite opinion set. Let 
		\begin{equation}\label{walshscore}
		\mathscr{I}(c,w)=\sum_{i=1}^\infty 2^{-i}(v_w(p_i)-c(p_i))^2.
		\end{equation}
		Then:
		\begin{enumerate}
			\item if $c$ is incoherent, then $c$ is strongly dominated relative to $\mathscr{I}$ by a coherent credence function;
			\item if $c$ is coherent, then $c$ is not weakly dominated relative to $\mathscr{I}$ by any credence function $d\neq c$.
		\end{enumerate}
	\end{theorem}
	\noindent Part 1 of this result follows from Theorem \ref{inftheorem1} by defining $\mathscr{I}$ in terms of the generalized quasi-additive Bregman divergence generated by $\{2^{-i}\}_{i=1}^\infty$ and 
	\[\mathfrak{d}(x,y)=x^2-y^2-2y(x-y)=\varphi(x)-\varphi(y)-\varphi'(y)(x-y),\]
	where $\varphi(x)=x^2$. Note that $\mathscr{I}(c,w)<\infty$ for all $c\in \mathcal{C}$ and $w\in W$ as $\sum_{i=1}^\infty 2^{-i}<\infty$.

	\section[The Countably Infinite Case: The Sufficiency of Coherence]{The Countably Infinite Case: The Sufficiency of Coherence}\label{countablecase2}
	Unlike coherent credence functions on finite opinion sets, coherent credence functions on countably infinite opinion sets can be strongly dominated.
	\begin{example}\label{problemswithinfinities}
		Let $\mathcal{F}=\{\{n\geq N:n\in \mathbb{N}\}:N\in \mathbb{N}\}$ be an opinion set over $\mathbb{N}$ (including zero). Let 
		\[c(\{n\geq N\})=\frac{1}{\sqrt{N+1}}.\]
		Then $c$ is coherent---in fact, countably coherent (see Definition \ref{countablecoherent})---but $\mathscr{I}(c,w)=\infty$ for all $w\in W$ when $\mathscr{I}$ is the generalized Brier score. So any omniscient credence function strongly dominates $c$.
	\end{example}
	In fact, the classic example of a merely finitely additive probability function---the 0-1 function defined on the finite-cofinite algebra over $\mathbb{N}$ taking value 0 on finite sets---restricts to a coherent dominated credence function. 
		\begin{example}\label{problemswithinfinites2}
		Let $\mathcal{F}=\{\{n\leq N:n\in \mathbb{N}\}:N\in \mathbb{N}\}$ be an opinion set over $\mathbb{N}$ (including zero). Let 
		\[c(\{n\leq N\})=0.\]
		Then $c$ is coherent---as well as finitely supported and not countably coherent---but $\mathscr{I}(c,w)=\infty$ for all $w\in W$ when $\mathscr{I}$ is the generalized Brier score. So any omniscient credence function strongly dominates $c$.
	\end{example}
	The goal of this section is to characterize the opinion sets and inaccuracy measures for which some variant of Theorem \ref{step3result} holds. We extend Theorem \ref{step3result} by proving dominance results for \textit{countably coherent} credence functions and using an opinion set compactification construction to transfer these results to merely coherent credence functions. At points, our results will only apply to the generalized Brier score. We conjecture that any such result extends to any generalized legitimate inaccuracy measure. In any case, this is a well motivated restriction since the Brier score has been defended by many---including \cite{horwich}, \cite{maher}, \cite{Joyce}, and \cite{leitgebi}---as being a particularly appropriate way to measure inaccuracy.
	
    A summary of the main results from Sections \ref{countablecase1}-\ref{countablecase2} can be found in Figure \ref{summary}. 
     
     \textit{Throughout the rest of Section \ref{countablecase2} we assume that the opinion set $\mathcal{F}$ is countably infinite, unless otherwise stated.}
	
	\subsection{Countable Coherence}
	We begin by introducing the notion of a \textit{countably coherent} credence function and establishing a characterization theorem regarding countable coherence on \textit{countably discriminating} opinion sets which extends a result of \cite{definetti}.
	\begin{definition}\label{CC}
		For an opinion set $\mathcal{F}\subseteq \mathcal{P}(W)$, we define an equivalence relation $\sim$ on $W$ such that $w\sim w'$ if and only if $\{p\in \mathcal{F}:w\in p\}=\{p\in \mathcal{F}:w'\in p\}$. We call the set of equivalence classes of $W$ the \textit{quotient of $W$ relative to $\mathcal{F}$}. If the quotient of $W$ relative to $\mathcal{F}$ is countable, then we call $\mathcal{F}$ \textit{countably discriminating}.
	\end{definition}
	\noindent Clearly, any countable opinion set over a countable set of worlds is countably discriminating. 
	
	The following characterization of the coherent credence functions on finite opinion sets is due to \cite{definetti}. Recall $\mathcal{V}_{\mathcal{F}}$ is the set of omniscient credence functions on $\mathcal{F}$, which is finite when $\mathcal{F}$ is finite.
	\begin{theorem}[\citealt{definetti}]\label{finitecoherentcharacterization}
		$c$ is a coherent credence function on a finite opinion set $\mathcal{F}$ if and only if there are $\lambda_w\in [0,1]$ with $\sum_{v_w\in \mathcal{V}_\mathcal{F}}\lambda_w=1$ such that
		\[c(p)=\sum_{v_w\in \mathcal{V}_{\mathcal{F}}}\lambda_w v_{w}(p)\]
		for all $p\in \mathcal{F}$.
	\end{theorem}
	
	Theorem \ref{finitecoherentcharacterization} is integral to \citeauthor{predd}'s proof that coherence is sufficient to avoid dominance in Theorem \ref{step3result}. We now show \citeauthor{definetti}'s characterization of the coherent credence functions on finite opinion sets extends to countably coherent credence functions on countably infinite opinion sets that are countably discriminating.
	
	\begin{definition}\label{sigmaalgebra}
		A \textit{$\sigma$-algebra} over $W$ is a subset $\mathcal{F}^*\subseteq \mathcal{P}(W)$ such that: 
		\begin{enumerate}
			\item $W\in \mathcal{F}^*$;
			\item if $\{p_i\}_{i=1}^\infty\subseteq \mathcal{F}^*$, then $\bigcup_{i=1}^\infty p_i\in \mathcal{F}^*$;
			\item if $p\in \mathcal{F}^*$, then $W\setminus p\in \mathcal{F}^*$.
		\end{enumerate}
	\end{definition}
	\begin{definition}\label{countablecoherent}
		Let a credence function $c$ be \textit{countably coherent} if $c$ extends to a countably additive probability function on a $\sigma$-algebra $\mathcal{F}^*$ containing $\mathcal{F}$.\footnote{\label{sigma}Note that if $c$ is countably coherent on $\mathcal{F}$, then $c$ extends to a countably additive probability function on $\sigma(\mathcal{F})$, the $\sigma$-algebra generated by $\mathcal{F}$.} That is, there is a $c^*:\mathcal{F}^*\to [0,1]$ such that:
		\begin{enumerate}
			\item $c^*(p)=c(p)$ for all $p\in \mathcal{F}$;
			\item $c^*(\bigcup_{i=1}^\infty p_i)=\sum_{i=1}^\infty c^*(p_i)$ for $\{p_i\}_{i=1}^\infty\subseteq \mathcal{F}^*$ with $p_i\cap p_j=\varnothing$ for $i\neq j$;
			\item $c^*(W)=1$.
		\end{enumerate}
		Otherwise, a credence function is \textit{countably incoherent}. 
	\end{definition}
	\begin{restatable}{proposition}{countablerep}\label{countablerep} Let $\mathcal{F}$ be a countably infinite opinion set that is countably discriminating (so $V_{\mathcal{F}}$ is countable). Then a credence function $c$ is countably coherent if and only if there are  $\lambda_{v_w}\in [0,1]$ with $\sum_{v_w\in \mathcal{V}_\mathcal{F}}\lambda_{v_w}=1$ such that
		\[c(p)=\sum_{v_w\in \mathcal{V}_{\mathcal{F}}} \lambda_{v_w} v_w(p)\]
		for all $p\in \mathcal{F}$.
	\end{restatable}
		\begin{proof}
			We adapt the proof of Proposition 1 in \citealt{predd}. Let $\mathcal{F}=~\{p_1,p_2,\ldots\}$. Let $\mathcal{X}$ be the collection of all nonempty sets of the form $\bigcap_{i=1}^\infty p_i^*$ where $p_i^*$ is either $p_i$ or $p_i^c$. Then $\mathcal{X}$ partitions $W$. Also, $\mathcal{X}$ is in bijection with $\mathcal{V}_{\mathcal{F}}$, the set of omniscient credence functions. 
		
		Indeed, let $f$ map $v_w$ to $\bigcap_{i=1}^\infty p_i^*$ where $p_i^*=p_i$ if $v_w(p_i)=1$ and $p_i^*=p_i^c$ otherwise. Then for each $w$, $w\in f(v_w)$ and so $f(v_w)\in \mathcal{X}$. Note $f$ is onto. Indeed, let $w\in \bigcap_{i=1}^\infty p_i^*$, where $\bigcap_{i=1}^\infty p_i^*\in \mathcal{X}$. Then $f(v_w)=\bigcap_{i=1}^\infty p_i^*$. Also, $f$ is injective. Indeed, assume $f(v_w)=f(v_{w'})$. Then \[f(v_w)=\bigcap_{i=1}^\infty p^1_i=\bigcap_{i=1}^\infty p^2_i= f(v_{w'})\]
		for $p_i^j=p_i$ or $p_i^j=p_i^c$ for all $i\in \mathbb{N}$ and $j\in\{1,2\}$. If $p^1_i\neq p^2_i$ for some $i$, then without loss of generality we may assume $p^1_i=p_i$ and $p^2_i=p_i^c$. So $w\in p_i^1$ but $w\notin p_i^2$ and thus $w\notin \bigcap_{i=1}^\infty p_i^2$. But $w\in \bigcap_{i=1}^\infty p_i^1$ by definition of $f$ and so $\bigcap_{i=1}^\infty p_i^1\neq \bigcap_{i=1}^\infty p_i^2$, which is a contradiction. It follows that $p^1_i=p^2_i$ for all $i$, but then by definition of $f$, this implies $v_w(p_i)=1$ if and only if $v_{w'}(p_i)=1$ for all $i$ and so $v_w=v_{w'}$.

		It is easy to see that since $\mathcal{F}$ is countably discriminating, $\mathcal{V}_{\mathcal{F}}$ is countable. It follows that $\mathcal{X}$ is countable. Enumerate the elements of $\mathcal{V}_{\mathcal{F}}$ and $\mathcal{X}$ by $v_{w_1}, v_{w_2},\ldots$ and $e_1,e_2,\ldots$, respectively, such that $f^{-1}(e_j)=v_{w_j}$. We have that $p_i$ is the disjoint union of $e_j$ such that $e_j\subseteq p_i$, or equivalently the $e_j$ where $f^{-1}(e_j)(p_i)=1$. Note i) for any countably additive probability function $\mu$ on a $\sigma$-algebra containing $\mathcal{F}$ (and thus containing $\mathcal{X}$) and any $p_i\in \mathcal{F}$:
		\[\mu(p_i)=\sum_{j=1}^\infty \mu(e_j) f^{-1}(e_j)(p_i). \]
		
		Now we prove the equivalence. Assume $c$ is countably coherent. By the definition of countable coherence, $c$ extends to a countably additive probability function $\mu$ on a $\sigma$-algebra containing $\mathcal{F}$. Then by i), 
		\[c(p_i)=\mu(p_i)=\sum_{j=1}^\infty \mu(e_j) f^{-1}(e_j)(p_i)\]
		for all $p_i\in \mathcal{F}$. But since $\mu(e_j)$ are non-negative and sum to $1$ (since the $e_j$'s partition $W$ and $\mu$ is a countably additive probability function), we have that $c$ has the form stated.
		
		Now assume $c(p_i)=\sum_{j=1}^\infty \lambda_jv_{w_j}(p_i)$ for all $i$ where $\sum_{j=1}^\infty \lambda_j=1$. Let $\sigma(\mathcal{F})$ be the smallest $\sigma$-algebra on $W$ containing $\mathcal{F}$. Then it is easy to check that the function on $\sigma(\mathcal{F})$ defined by $\bar{v}_{w_j}(p)=1$ if and only if $w_j\in p$ extends $v_{w_j}$ and is a countably additive probability function on $\sigma(\mathcal{F})$. Then $\sum_{j=1}^\infty \lambda_j \bar{v}_{w_j}$ is a countably additive probability function on $\sigma(\mathcal{F})$ since a countable sum of countably additive probability functions with coefficients that sum to $1$ is a countably additive probability function. Since 
		\[c(p_i)= \sum_{j=1}^\infty \lambda_jv_{w_j}(p_i)=\sum_{j=1}^\infty \lambda_i \bar{v}_{w_j}(p_i)\]
		for all $i$, it follows that $c$ extends to a countably additive probability function on a $\sigma$-algebra containing $\mathcal{F}$.
	\end{proof}

	\subsection{Compactification of an Opinion Space}\label{compactsection}
	In this section, we introduce the compactification construction of what we call an \textit{opinion space}. The construction will be relevant to transferring dominance results for countably coherent credence functions to merely coherent credence functions, the reason being that merely coherent credence functions become countably coherent if the underlying set of worlds is ``compactified''.
	\begin{definition}
		An \textit{opinion space} is a pair $(W,\mathcal{F})$, where $W$ is a nonempty set and $\mathcal{F}\subseteq \mathcal{P}(W)$.
	\end{definition}
	\noindent From here on out we will speak in terms of opinion spaces as opposed to opinion sets in order to keep track of the underlying set of worlds. We continue to assume that $\mathcal{F}$ is countably infinite.
	
	\cite{borkar} proved that the opinion spaces which satisfy a certain compactness property are precisely those where the set of coherent credence functions and the set of countably coherent credence functions coincide.\footnote{\citeauthor{borkar} do not restrict attention to opinion spaces where $\mathcal{F}$ is countably infinite, and it is easy to see that there are analogues of the following compactification results in the uncountable setting. However, we continue to restrict attention to the countably infinite setting since we have not yet extended the accuracy framework beyond that.}
	\begin{definition}\label{compactdef}
		Let $(W,\mathcal{F})$ be an opinion space. Let $f(n)\in \{0,1\}$ and set $p_n^{f(n)}=p_n$ if $f(n)=0$ and $p_n^{f(n)}=p_n^c$ if $f(n)=1$. Then $(W,\mathcal{F})$ is \textit{compact} if for any choice of $\{p_n\}_{n=1}^\infty\subseteq \mathcal{F}$ and $f:\mathbb{N}\to \{0,1\}$, if $\bigcap_{n=1}^Np_n^{f(n)}$ is nonempty for every $N$, then $\bigcap_{n=1}^\infty p_n^{f(n)}$ is nonempty.
	\end{definition}
	\noindent As an example, note that the opinion spaces from Examples \ref{problemswithinfinities} and \ref{problemswithinfinites2} are not compact. Indeed, for the first example $\bigcap_{n=1}^\infty p_n=\varnothing$ and yet every finite subset of $\mathcal{F}$ has nonempty intersection; for the second example, $\bigcap_{n=1}^\infty p_n^c=\varnothing$ while $\bigcap_{n=1}^Np_n^c\neq\varnothing$ for every $N$.
% 	\begin{remark}
% 		Assume $\mathcal{F}$ is closed under finite intersections. Let $\mathcal{A}(\mathcal{F})$ denote the algebra generated by $\mathcal{F}$, and let $\mathcal{T}(\mathcal{F})$ denote the topology generated by $\mathcal{F}$. By the Alexander subbase theorem (see, e.g., \citealt[p.~139]{kelley}), $(W,\mathcal{F})$ is compact if and only if $\mathcal{T}(\mathcal{A}(\mathcal{F}))$ is compact.\footnote{Proof: The set of elements in $\mathcal{F}$ and their complements form a subbase for $\mathcal{T}(\mathcal{A}(\mathcal{F}))$.}
% 	\end{remark}
	\begin{theorem}[\citealt{borkar}]\label{borkerresult}The following are equivalent:
	\begin{enumerate}
	    \item $(W,\mathcal{F})$ is compact;
	\item for every credence function $c$ on $(W,\mathcal{F})$, $c$ is coherent if and only if $c$ is countably coherent.
	\end{enumerate}
	\end{theorem}

	We now show how to turn any space into a compact space and, in light of Theorem \ref{borkerresult}, any coherent credence function into a countably coherent credence function. Let $(W,\mathcal{F})$ be an opinion space. Let $S$ denote the set of sequences of the form $\{p_n^{f(n)}\}$ (as in Definition \ref{compactdef}) such that $\bigcap_{n=1}^Np_n^{f(n)}\neq \varnothing$ for every $N$ but  $\bigcap_{n=1}^\infty p_n^{f(n)}=\varnothing$. Define $W^*=W\cup \{x_s:s\in S\}$, where each $x_s$ is a formal point corresponding to the element $s\in S$. Define $\mathcal{F}^*\subseteq \mathcal{P}(W^*)$ as follows: for each $p\in \mathcal{F}$, let $S_p$ denote the set of sequences $s$ of the form $\{p_n^{f(n)}\}$ (as in Definition \ref{compactdef}) such that $s\in S$, $p_n=p$ for some $n$, and $f(n)=0$. Then define
	\[p^*=p\cup \{x_s:s\in S_p\}.\]
	Finally, let $\mathcal{F}^*=\{p^*:p\in \mathcal{F}\}$. We call $(W^*, \mathcal{F}^*)$ the \textit{compactification} of $(W,\mathcal{F})$. We always denote the compactification of $(W,\mathcal{F})$ by $(W^*,\mathcal{F}^*)$. Further, we let $\Psi$ denote the natural bijection from $\mathcal{F}$ to $\mathcal{F}^*$ given by $\Psi(p)=p^*$.

	We first note that $(W^*,\mathcal{F}^*)$ is in fact compact.
	\begin{restatable}{lemma}{iscompact}\label{iscompact} For $(W,\mathcal{F})$ an opinion space,  $(W^*,\mathcal{F}^*)$ is compact.
	\end{restatable}
		\begin{proof}
		Let $\{\Psi(p_n)^{f(n)}\}_{n=1}^\infty$ be a sequence of elements of $\mathcal{F}^*$ or their complements as in Definition \ref{compactdef}. Case 1: for each $N$ there is some $w_N\in W$ such that $w_N\in \bigcap_{n=1}^N\Psi(p_n)^{f(n)}$. Then since i) $\Psi(p)\cap W=p$ and ii) $\Psi(p)^c\cap W =p^c$ for any $p\in \mathcal{F}$, it follows that $w_N\in \bigcap_{n=1}^N p^{f(n)}_n$ for each $N$. If there is some $w'\in W$ with $w'\in \bigcap_{n=1}^\infty p^{f(n)}_n$ then by i) and ii) it follows that $w'\in \bigcap_{n=1}^\infty\Psi(p_n)^{f(n)}$. Otherwise, by construction, we defined some $x_s$ to be such that $x_s \in \bigcap_{n=1}^\infty \Psi(p_n)^{f(n)}$. In either case, we are done. Case 2: there is some $N$ such that $\bigcap_{n=1}^N\Psi(p_n)^{f(n)}\subseteq W^*\setminus W$. I claim this implies that $\bigcap_{n=1}^N\Psi(p_n)^{f(n)}=\varnothing$. Indeed, if there were some $w\in W^*\setminus W$ such that $w\in\bigcap_{n=1}^N\Psi(p_n)^{f(n)}$, then that is because $\{p_n^{f(n)}\}_{n=1}^N$ is an initial sequence of some sequence $\{\bar{p}_n^{\bar{f}(n)}\}_{n=1}^\infty$  such that $\bigcap_{n=1}^l\bar{p}^{\bar{f}(n)}_n\neq \varnothing$ for each $l$ and thus, in particular, $\bigcap_{n=1}^Np^{f(n)}_n\neq \varnothing$. So there is some $w\in W$ such that $w\in \bigcap_{n=1}^N\Psi(p_n)^{f(n)}$ by i) and ii), which is a contradiction. Thus we have established that $(W^*,\mathcal{F}^*)$ is compact.
	\end{proof}

		Next we note that, as suggested, we can naturally turn a coherent credence function into a countably coherent credence function by compactifying the underlying opinion space.
		
% 		$(W^*,\mathcal{F}^*)$. Intuitively, we can turn a merely coherent credence function into a countably coherent credence function by ``filling in holes'' in the underlying space of possible worlds.
	
	\begin{restatable}{lemma}{connectinglemma}\label{connectinglemma}
		Let $(W,\mathcal{F})$ be an opinion space and $c$ a coherent credence function on $(W,\mathcal{F})$. Let $(W^*,\mathcal{F}^*)$ be the compactification of $(W,\mathcal{F})$ and define $c^*(\Psi(p)):=c(p)$ for each $p\in \mathcal{F}$. Then $c^*$ is a countably coherent credence function on $(W^*,\mathcal{F}^*)$ and $\mathscr{I}(c,w)=\mathscr{I}(c^*,w)$ for $w\in W$.
	\end{restatable}
		\begin{proof}
		Since $(W^*,\mathcal{F}^*)$ is compact, we only need to show that $c^*$ is coherent by Theorem \ref{borkerresult}. Thus it suffices to show that $c^*$ can be extended to a finitely additive probability function on $\mathcal{A}(\mathcal{F}^*)$. Since $c$ is coherent, there is a finitely additive probability function $\bar{c}$ such that:
		\begin{enumerate}
			\item $\bar{c}(p)=c(p)$ for $p\in \mathcal{F}$;
			\item $\bar{c}(p\cup q)=\bar{c}(p)+\bar{c}(q)$ for $p,q\in \mathcal{F}$ with $p\cap q=\varnothing$;
			\item $\bar{c}(W)=1$.
		\end{enumerate}
		First, define $\Psi(p^c):=\Psi(p)^c$ for each $p\in \mathcal{F}$. Then each element in $\mathcal{A}(\mathcal{F}^*)$ can be represented by $\bigcup_{i=1}^{N}\bigcap_{j=1}^M\Psi(q_{ij})$ where $q_{ij}$ or its complement is in $\mathcal{F}$. We define
		\[\bar{c^*}(\bigcup_{i=1}^{N}\bigcap_{j=1}^M\Psi(q_{ij})):=\bar{c}(\bigcup_{i=1}^{N}\bigcap_{j=1}^Mq_{ij}).\]
		Using that $p=\Psi(p)\cap W$ and $p^c=\Psi(p)^c\cap W$, we show that $\bar{c^*}$ is a well-defined finitely additive probability function on $\mathcal{A}(\mathcal{F}^*)$ extending $c^*$. We first show $\bar{c^*}$ is well-defined. Assume that 
		\[\bigcup_{i=1}^{N}\bigcap_{j=1}^M\Psi(q_{ij})=\bigcup_{i=1}^{N'}\bigcap_{j=1}^{M'}\Psi(r_{ij}).\]
		Then this clearly implies that 
		\[\bigcup_{i=1}^{N}\bigcap_{j=1}^M\Psi(q_{ij})\cap W=\bigcup_{i=1}^{N'}\bigcap_{j=1}^{M'}\Psi(r_{ij})\cap W\]
		which, noting that $p=\Psi(p)\cap W$ and $p^c=\Psi(p)^c\cap W$, establishes that 
		\[\bigcup_{i=1}^{N}\bigcap_{j=1}^Mq_{ij}=\bigcup_{i=1}^{N'}\bigcap_{j=1}^{M'}r_{ij},\]
		and so 
		\[\bar{c^*}(\bigcup_{i=1}^{N}\bigcap_{j=1}^M\Psi(q_{ij}))=\bar{c}(\bigcup_{i=1}^{N}\bigcap_{j=1}^Mq_{ij})=\bar{c}(\bigcup_{i=1}^{N'}\bigcap_{j=1}^{M'}r_{ij})=\bar{c^*}(\bigcup_{i=1}^{N'}\bigcap_{j=1}^{M'}\Psi(r_{ij})).\]
		Thus $\bar{c^*}$ is well-defined. Clearly, $\bar{c^*}$ extends $c^*$. Now, since $W\subseteq W^*$, if  
		\[\bigcup_{i=1}^{N}\bigcap_{j=1}^M\Psi(q_{ij})\cap \bigcup_{i=1}^{N'}\bigcap_{j=1}^{M'}\Psi(r_{ij})=\varnothing\]
		then 
		\[\bigcup_{i=1}^{N}\bigcap_{j=1}^M\Psi(q_{ij})\cap W\cap \bigcup_{i=1}^{N'}\bigcap_{j=1}^{M'}\Psi(r_{ij})\cap W=\varnothing\]
		and so
		\[\bar{c}(\bigcup_{i=1}^{N}\bigcap_{j=1}^Mq_{ij}\cup \bigcup_{i=1}^{N'}\bigcap_{j=1}^{M'}r_{ij})= \bar{c}(\bigcup_{i=1}^{N}\bigcap_{j=1}^Mq_{ij})+ \bar{c}(\bigcup_{i=1}^{N'}\bigcap_{j=1}^{M'}r_{ij}).\]
		Then noting the definition of $\bar{c^*}$ in terms of $\bar{c}$, we establish finite additivity. Lastly, if \[W^*=\bigcup_{i=1}^{N}\bigcap_{j=1}^M\Psi(q_{ij}),\]
		then 
		\[W=\bigcup_{i=1}^{N}\bigcap_{j=1}^M\Psi(q_{ij})\cap W,\]
		and so 
		\[\bar{c^*}(\bigcup_{i=1}^{N}\bigcap_{j=1}^M\Psi(q_{ij}))=\bar{c}(\bigcup_{i=1}^{N}\bigcap_{j=1}^Mq_{ij})=\bar{c}(W)=1.\]
		This establishes that $c^*$ is coherent on $(W^*,\mathcal{F}^*)$, and so since $(W^*,\mathcal{F}^*)$ is compact, $c^*$ is countably coherent. Further, $w\in p$ if and only if $w\in \Psi(p)$ for each $w\in W$, so $v_w$ defined on $\mathcal{F}$ is the same as $v_w$ defined on $\mathcal{F}^*$ for each $w\in W$. Since $c(p)=c^*(\Psi(p))$ for all $p \in \mathcal{F}$, this establishes that $\mathscr{I}(c,w)=\mathscr{I}(c^*,w)$ for each $w\in W$.
	\end{proof}

\noindent For a coherent credence function $c$ defined on an opinion space $(W,\mathcal{F})$, we let $c^*$ denote the countably coherent credence function on $(W^*,\mathcal{F}^*)$ given as in Lemma \ref{connectinglemma}. 
    
    \begin{example}\label{examplecomputation}
	As an example, let us compute the compactification of the opinion space from Example \ref{problemswithinfinites2} and show how to identify a coherent credence function on the space with a countably coherent credence function on its compactification. We note that only for $f(n)=1$ for all $n\in \mathbb{N}$ is $\bigcap_{n=1}^Np_n^{f(n)}$ nonempty for every $N$ while $\bigcap_{n=1}^\infty p_n^{f(n)}=\varnothing$. Indeed, assume $f(m)=0$ for some $m$. If $f(i)=1$ for some $i\geq m+1$, then since $p_i^c\cap p_m=\varnothing$, we have $\bigcap_{n=1}^ip_n^{f(n)}=\varnothing$ which contradicts our assumption. So $f(i)=0$ for all $i\geq m+1$. But then since $\bigcap_{n=1}^mp_n^{f(n)}\neq \varnothing$ and  $\bigcap_{n=1}^mp_n^{f(n)}\subseteq p_i$ for all $i\geq m+1$, it also follows that $\bigcap_{n=1}^\infty p_n^{f(n)}\neq \varnothing$, which contradicts our assumption. So $S$ is a single point $x$, $W^*=W\cup\{w^*\}$, and $\mathcal{F}^*=\{\{n\leq N\}:N\in \mathbb{N}\}$. $\mathcal{F}^*$ is identical to $\mathcal{F}$, except that there is a point in the complement of every proposition in $\mathcal{F}^*$. For a coherent credence function $c$ on $(W,\mathcal{F})$, $c^*$ on $(W^*,\mathcal{F}^*)$ is identical to $c$ and is a countably coherent credence function on the compact opinion space $(W^*,\mathcal{F}^*)$. For example, for credence function $c$ in Example \ref{problemswithinfinites2}, $c^*$ extends to the countably additive omniscient credence function $v_{w^*}$ on the $\sigma$-algebra generated by $\mathcal{F}^*$.
		    \end{example}
		Using Theorem \ref{borkerresult}, Lemma \ref{iscompact}, and Lemma \ref{connectinglemma}, the proof strategy for extending Theorem \ref{step3result} is more precisely as follows. First, we establish dominance results for countably coherent credence functions. Second, we transform each coherent credence function on $(W,\mathcal{F})$ into a countably coherent credence function on $(W^*,\mathcal{F}^*)$ as in Lemma \ref{connectinglemma}. Lastly, we use the dominance results for countably coherent credence functions to establish dominance results for coherent credence functions in certain cases where there is ``accuracy dominance stability'' in compactifying.

	\subsection{W-Stable Opinion Spaces}\label{wstabilitysection}
	In this section, we establish the equivalence between coherence and avoiding weak dominance for certain countably infinite opinion spaces (Theorem \ref{weakdominancemainresult}), as well as additional results extending Theorem \ref{step3result} to the countably infinite setting (Corollary \ref{corollarytoweak} and Theorem \ref{brierequivalence}). We first note that under certain circumstances countably coherent credence functions are not weakly dominated (Proposition \ref{CCresult} and Proposition \ref{pointfiniteresult}); then we use the compactification construction from the previous section and a property of an opinion space---\textit{W-stability} (Definition \ref{wstability})---to establish that for certain opinion spaces, mere coherence is also sufficient to avoid weak dominance.
	
    We first prove that if a countably coherent credence function $c$ has finite expected inaccuracy, then $c$ is not weakly dominated.
	\begin{definition}\label{finiteexpectedinaccuracy}
		For $c$ a countably coherent credence function and $\mathscr{I}$ a generalized legitimate inaccuracy measure, we say that $c$ has \textit{finite expected inaccuracy relative to} $\mathscr{I}$ if $c$ has a countably additive extension $\bar{c}$ defined on the opinion space $(W,\sigma(\mathcal{F}))$ such that $\mathbb{E}_{\bar{c}}\mathscr{I}(c,\cdot)<\infty$.\footnote{Consider the measure space $(W,\sigma(\mathcal{F}),\mu)$. Note $\mathfrak{d}(v_{w}(p_i),d_i)=1_{p_i}(w)\mathfrak{d}(1,d_i)+(1-1_{p_i}(w))\mathfrak{d}(0,d_i)$ so that each term in $\mathscr{I}(d,\cdot)$ is measurable for any credence function $d$, and so the infinite sum is measurable as the finite sum and limit of measurable functions are measurable. Thus we can take the expectation of $\mathscr{I}(d,\cdot)$ with respect to $\mu$ for any credence function $d$.} For $c$ a coherent but not countably coherent credence function and $\mathscr{I}$ a generalized legitimate inaccuracy measure, we say that $c$ has \textit{finite expected inaccuracy relative to} $\mathscr{I}$ if $c^*$ has finite expected inaccuracy relative to $\mathscr{I}$.
	\end{definition}
% \noindent    Note that it follows by Definition \ref{finiteexpectedinaccuracy} that any coherent credence function $c$ has finite expected inaccuracy if and only if $c^*$ has finite expected inaccuracy. 
	\begin{restatable}{proposition}{CCresult}\label{CCresult}
		Let $(W,\mathcal{F})$ be an opinion space and $\mathscr{I}$ a generalized legitimate inaccuracy measure.	If $c$ is a countably coherent credence function on $(W, \mathcal{F})$ with finite expected inaccuracy relative to $\mathscr{I}$, then $c$ is not weakly dominated relative to $\mathscr{I}$.
	\end{restatable}
		\begin{proof}
		Since $c$ is countably coherent, let $\bar{c}$ be a countably additive probability function on $\sigma(\mathcal{F})$ extending $c$ such that $\mathbb{E}_{\bar{c}}\mathscr{I}(c,\cdot)<\infty$.  Note that $\mathfrak{d}$ is a \textit{strictly proper inaccuracy measure} for singleton opinion sets (see Remark \ref{strictlyproperremark} and \citealt[Theorem~4.3.5]{pettigrew} for a precise definition), which implies by definition that $p\mathfrak{d}(1,x)+(1-p)\mathfrak{d}(0,x)$ is uniquely minimized at $x=p$.\footnote{In more detail, a special case of Theorem 4.3.5 in \citealt{pettigrew} establishes that since $\mathfrak{d}$ is a one-dimensional Bregman divergence, there is a function $s:\{0,1\}\times [0,1]\to [0,\infty]$ that satisfies i) $ps(1,x)+(1-p)s(0,x)$ is uniquely minimized at $x=p$ for all $p\in [0,1]$, ii) $s$ is continuous, and iii) $s(v_w,x)=d(v_w,x)$ for all $w\in W$ and $x\in [0,1]$.} It follows that \[\mathbb{E}_{\bar{c}}\mathfrak{d}(v_w,c_i)=c_i\mathfrak{d}(1,c_i)+(1-c_i)\mathfrak{d}(0,c_i)<c_i\mathfrak{d}(1,x)+(1-c_i)\mathfrak{d}(0,x)=\mathbb{E}_{\bar{c}}\mathfrak{d}(v_w,x)\]
		for any $x\neq c_i$.
		
		Assume toward a contradiction that there is a credence function $d$ with $d\neq c$ and $\mathscr{I}(d,w)\leq \mathscr{I}(c,w)$ for each $w$ with strict inequality for some $w$. Then $\mathbb{E}_{\bar{c}}\mathscr{I}(d,\cdot)\leq \mathbb{E}_{\bar{c}}\mathscr{I}(c,\cdot)<\infty$, so both $\mathscr{I}(d,\cdot)$ and $\mathscr{I}(c,\cdot)$ are integrable with respect to the measure space $(W,\sigma(\mathcal{F}),\bar{c})$. Then let $i$ be any index such that $d_i\neq c_i$. There must be at least one since $c\neq d$. Then  
		\[\mathbb{E}_{\bar{c}}\mathfrak{d}(v_w,c_i)< \mathbb{E}_{\bar{c}}\mathfrak{d}(v_w,d_i).\]
		If $i$ is such that $d_i=c_i$ then clearly $\mathbb{E}_{\bar{c}}\mathfrak{d}(v_w,c_i)= \mathbb{E}_{\bar{c}}\mathfrak{d}(v_w,d_i)$. So since $\mathbb{E}_{\bar{c}}\mathscr{I}(c,\cdot)<\infty$ and $\mathbb{E}_{\bar{c}}\mathscr{I}(d,\cdot)<\infty$, we have
		\[\mathbb{E}_{\bar{c}}\mathscr{I}(c,\cdot)=\sum_{i=1}^\infty a_i\mathbb{E}_{\bar{c}}\mathfrak{d}(v_w,c_i)<\sum_{i=1}^\infty a_i\mathbb{E}_{\bar{c}}\mathfrak{d}(v_w,d_i)=\mathbb{E}_{\bar{c}}\mathscr{I}(d,\cdot),\]
		which implies that $\mathbb{E}_{\bar{c}}(\mathscr{I}(c,\cdot)-\mathscr{I}(d,\cdot))<0$. Thus there is some nonempty set $E\in \sigma(\mathcal{F})$ with $\bar{c}(E)>0$ on which $\mathscr{I}(c,\cdot)-\mathscr{I}(d,\cdot)<0$ (since the Lebesgue integral is positive). But this contradicts our assumption that $d$ weakly dominates $c$, and so we are done.
	\end{proof}

	Here is another dominance result for countably coherent credence functions where we assume $\mathcal{F}$ is \textit{point-finite} $(|\{p\in \mathcal{F}:w\in p\}|<\infty$ for all $w\in W$) but weaken the assumption that $c$ has finite expected inaccuracy considerably, namely to \textit{somewhere finitely inaccurate} (there is a $w\in W$ such that $\mathscr{I}(c,w)<\infty$). We also restrict to the generalized Brier score $\mathscr{B}$.
	\begin{restatable}{proposition}{pointfiniteresult}\label{pointfiniteresult}
	Let  $(W,\mathcal{F})$ be a point-finite opinion space. If a credence function $c$ is countably coherent and somewhere finitely inaccurate relative to $\mathscr{B}$, then $c$ is not weakly dominated relative to $\mathscr{B}$.
	\end{restatable}

	\begin{proof}
		Assume $d$ weakly dominates $c$. Note i) $c$ is somewhere finitely inaccurate if and only if $\mathscr{B}(c,w)<\infty$ for all $w\in W$ if and only if $\sum_{i=1}^\infty c_i^2<\infty$. It follows by weak dominance that $\mathscr{B}(d,w)<\infty$ for all $w\in W$ and therefore $\sum_{i=1}^\infty d_i^2<\infty$. Let $\mathscr{B}(c,w)=\mathfrak{D}(v_w,c)$ for $\mathfrak{D}$ a generalized quasi-additive Bregman divergence.
		
		 Since $(W,\mathcal{F})$ is point-finite, it is also countably discriminating as there are only countably many finite subsets of $\mathcal{F}$. So by Proposition \ref{countablerep}, $c=\sum_{j=1}^\infty \lambda_j v_{w_j}$ for $\lambda_j\in [0,1]$ with $\sum_{j=1}^\infty \lambda_j=1$ and $\mathcal{V}_{\mathcal{F}}=\{v_{w_j}\}_{j=1}^\infty$. First, note that $\mathfrak{D}(\sum_{j=1}^\infty \lambda_jv_{w_j},c)=0$ and $(\sum_{j=1}^\infty \lambda_jv_{w_j}(p_i)-c(p_i))^2=0$ for all $i$, so
		\[\mathfrak{D}(\sum_{j=1}^\infty \lambda_jv_{w_j},c)-\mathfrak{D}(\sum_{j=1}^\infty \lambda_jv_{w_j},d)=\sum_{i=1}^\infty a_i[(\sum_{j=1}^\infty \lambda_jv_{w_j}(p_i)-c_i)^2-(\sum_{j=1}^\infty \lambda_jv_{w_j}(p_i)-d_i)^2].\]
		 
		Using that
		 \begin{align*}
		 (\sum_{j=1}^\infty \lambda_jv_{w_j}(p_i)-c_i)^2-(\sum_{j=1}^\infty \lambda_jv_{w_j}(p_i)-d_i)^2 & =c_i^2-d_i^2+2(d_i-c_i)\sum_{j=1}^\infty \lambda_jv_{w_j}(p_i)\\
		 &=\sum_{j=1}^\infty \lambda_j c_i^2-\sum_{j=1}^\infty \lambda_j d_i^2+\sum_{j=1}^\infty \lambda_j2(d_i-c_i)v_{w_j}(p_i)\\
		 &=\sum_{j=1}^\infty \lambda_j [c_i^2- d_i^2+2(d_i-c_i)v_{w_j}(p_i)]\\
		&= \sum_{j=1}^\infty \lambda_j[(v_{w_j}(p_i)-c_i)^2-(v_{w_j}(p_i)-d_i)^2]
		 \end{align*}
		 for each $i$ since $\sum_{j=1}^\infty\lambda_j=1$, we have that 
		\begin{align}
		\mathfrak{D}(\sum_{j=1}^\infty \lambda_jv_{w_j},c)-\mathfrak{D}(\sum_{j=1}^\infty \lambda_jv_{w_j},d)&=\sum_{i=1}^\infty a_i\sum_{j=1}^\infty \lambda_j [(v_{w_j}(p_i)-c_i)^2-(v_{w_j}(p_i)-d_i)^2]\label{switchequation}\\
		&=\sum_{i=1}^\infty a_i(\sum_{j:w_j\notin p_i} \lambda_j) (c_i^2-d_i^2)
		+a_i(\sum_{j:w_j\in p_i}\lambda_j)((1-c_i)^2-(1-d_i)^2)\notag\\
		&=\sum_{i=1}^\infty a_i( c_i^2-d_i^2)+2a_i(\sum_{j:w_j\in p_i}\lambda_j)(d_i-c_i)\notag\\
		&=\sum_{i=1}^\infty a_i( -c_i^2-d_i^2)+2a_i(\sum_{j:w_j\in p_i}\lambda_j)d_i\notag
		\end{align}
		since $c_i=\sum_{j:w_j\in p_i}\lambda_j$. We have $\sum_{i=1}^\infty c_i^2+d_i^2<\infty$ by i). Thus 
		\begin{equation}\label{finitelinearterm}
		0\leq \sum_{i=1}^\infty a_i2(\sum_{j:w_j\in p_i}\lambda_j)d_i<\infty
		\end{equation}
		because 	
		\[0\geq \mathfrak{D}(\sum_{j=1}^\infty \lambda_jv_{w_j},c)-\mathfrak{D}(\sum_{j=1}^\infty \lambda_jv_{w_j},d).\] 
		
		Having established (\ref{finitelinearterm}), we claim we can use the dominated convergence theorem (see, e.g., Theorem 1.4.49 in \citealt{tao}) to switch limits in (\ref{switchequation}). Indeed,	
		\[	\sum_{i=1}^\infty a_i \sum_{j=1}^N \lambda_j[(v_{w_j}(p_i)-c_i)^2-(v_{w_j}(p_i)-d_i)^2]
		=\sum_{i=1}^\infty a_i(\sum_{1\leq j\leq N}\lambda_j)(c_i^2-d_i^2)+2a_i(\sum_{\substack{j:w_j\in p_i\\1\leq j\leq N}}\lambda_j)(d_i-c_i).\]
		Letting 
		\[g_N(i)=a_i(\sum_{1\leq j\leq N}\lambda_j)(c_i^2-d_i^2)+2a_i(\sum_{\substack{j:w_j\in p_i\\1\leq j\leq N}}\lambda_j)(d_i-c_i)\]
		and noting that $-(\sum_{\substack{j:w_j\in p_i\\1\leq j\leq N}}\lambda_j)c_i\geq- c_i^2$ since $c_i=\sum_{j:w_j\in p_i}\lambda_j$, we see that 
		\[|g_N(i)|\leq a_i( 2c_i^2+d_i^2+2(\sum_{j:w_j\in p_i}\lambda_j)d_i).\] 
		Each of $c_i^2,d_i^2,$ and $(\sum_{j:w_j\in p_i}\lambda_j)d_i$ is summable in $i$ and $\sup_i a_i<\infty$. So, the dominated convergence theorem applies, and we can switch limits.
		 
		Thus we have 
		\begin{align*}
		0&\geq \mathfrak{D}(\sum_{j=1}^\infty \lambda_jv_{w_j},c)-\mathfrak{D}(\sum_{j=1}^\infty \lambda_jv_{w_j},d)\\
		&= \sum_{j=1}^\infty \lambda_j\sum_{i=1}^\infty a_i[(v_{w_j}(p_i)-c(p_i))^2-(v_{w_j}(p_i)-d(p_i))^2] \\
		&=\sum_{j=1}^\infty \lambda_j (\mathfrak{D}(v_{w_j},c)-\mathfrak{D}(v_{w_j},d))\geq 0
		\end{align*}
	where we used that $\mathfrak{D}(v_{w_j},c)=\mathscr{B}(c,w_j)<\infty$ and $\mathfrak{D}(v_{w_j},d)=\mathscr{B}(d,w_j)<\infty$ for each $j$ by i) to break up the summation in the second line. Thus we conclude that $c=d$, as $\mathfrak{D}(c,d)=0$ if and only if $c=d$.
	\end{proof}
	
	We now introduce the notion of \textit{W-stability} which will allow us to use Propositions \ref{CCresult} and \ref{pointfiniteresult} to prove extensions of Theorem \ref{step3result}. Intuitively, an opinion space is W-stable if a coherent credence function $c$ and its countably coherent counterpart $c^*$ are weakly dominated in precisely the same cases. Thus, whether a 
	credence function on a W-stable opinion space is weakly dominated does not depend on whether the underlying opinion space on which it is defined is compactified.
	\begin{definition}\label{wstability}
		Let $(W,\mathcal{F})$ be \textit{W-stable relative to }$\mathscr{I}$ if for any coherent credence function $c$ on $(W,\mathcal{F})$, if $c$ is weakly dominated relative to $\mathscr{I}$, then $c^*$ on $(W^*,\mathcal{F}^*)$ is  weakly dominated relative to $\mathscr{I}$.\footnote{It is easy to see that for any opinion space and inaccuracy measure, if $c^*$ is weakly dominated, then $c$ is weakly dominated.}
	\end{definition}
		\begin{remark}
	Not all opinion spaces are W-stable relative to every generalized legitimate inaccuracy measure. Indeed, consider the opinion space from Example \ref{problemswithinfinites2}. Then the credence function $c$ in that example which assigns $0$ to each proposition is strongly dominated. However, where $w^*$ is the single world added to the complement of each proposition when compactifying (see Example \ref{examplecomputation}) and $v_{w^*}:\mathcal{F}^*\to [0,1]$ is the omniscient credence function at world $w^*$, we have that $c^*=v_{w^*}$ and so $\mathscr{B}(c^*,w^*)=0$. Further, for any credence function $d\neq c^*$, we have $\mathscr{B}(d,w^*)>0$ since $\mathscr{B}(d,w^*)=0$ if and only if $d=v_{w^*}=c^*$. Thus, once the underlying space is compactified, $c$ is no longer weakly dominated relative to $\mathscr{B}$. 
	\end{remark}
	Using Proposition \ref{CCresult}, we establish one of our main results: sufficient and partly necessary conditions on an opinion space for coherence to be equivalent to not being weakly dominated.
	\begin{theorem}\label{weakdominancemainresult}
		Let $\mathscr{I}$ be a generalized legitimate inaccuracy measure and $(W,\mathcal{F})$ a W-stable opinion space relative to  $\mathscr{I}$ where all coherent credence functions have finite expected inaccuracy relative to $\mathscr{I}$. Then the following are equivalent:
		\begin{enumerate}
			\item $c$ is coherent;
			\item $c$ is not weakly dominated.
		\end{enumerate}
	\end{theorem}
	\begin{proof}
		We prove that if $c$ is coherent, then $c$ is not weakly dominated. Let $(W^*,\mathcal{F}^*)$ be the compactification of $(W,\mathcal{F})$. If $c$ is coherent on $(W,\mathcal{F})$, then $c^*$ is countably coherent by Lemma \ref{connectinglemma}. Further $c^*$ has finite expected inaccuracy by definition and the assumption that $c$ has finite expected inaccuracy. So by Proposition \ref{CCresult}, $c^*$ is not weakly dominated. But since $(W,\mathcal{F})$ is W-stable this implies that $c$ is not weakly dominated. The other direction follows from Theorem~\ref{inftheorem1}.
	\end{proof}
	\begin{remark}
		It is trivial to see that W-stability is necessary for the equivalence of coherence and not being weakly dominated. It is open how far finite expected inaccuracy can be weakened. 
	\end{remark}
	\begin{remark}
		If $\mathscr{I}$ is defined with summable weights, that is, $\{a_i\}_{i=1}^\infty$ such that $\sum_{i=1}^\infty a_i<\infty$, then there is a $C<\infty$ such that $\mathscr{I}(c,w)<C$ for all credence functions $c$ and $w\in W$. So, in particular, all coherent credence functions have finite expected inaccuracy relative to $\mathscr{I}$.
	\end{remark}
	If we add in an additional finiteness assumption, then we get the full equivalence of Theorem~\ref{step3result}.
	
	\begin{corollary}\label{corollarytoweak}
		In Theorem \ref{weakdominancemainresult}, if in addition all coherent credence functions $c$ have $\mathscr{I}(c,w)<\infty$ for all $w\in W$, then the following are equivalent:
		\begin{enumerate}
			\item $c$ is coherent;
			\item $c$ is not weakly dominated;
			\item $c$ is not strongly dominated.
		\end{enumerate}
	\end{corollary}
	We combine W-stability and Proposition \ref{pointfiniteresult} to get another set of sufficient conditions on $(W,\mathcal{F})$ for Theorem \ref{step3result} to go through for the generalized Brier score.
	
	\begin{restatable}{theorem}{brierequivalence}\label{brierequivalence}
		Let $(W,\mathcal{F})$ be a W-stable opinion space with $(W^*,\mathcal{F}^*)$ point-finite such that all coherent credence functions on  $(W,\mathcal{F})$ are somewhere finitely inaccurate relative to $\mathscr{B}$. Then the following are equivalent:
		\begin{enumerate}
			\item $c$ is coherent;
			\item $c$ is not weakly dominated relative to $\mathscr{B}$;
			\item $c$ is not strongly dominated relative to $\mathscr{B}$.
		\end{enumerate}
	\end{restatable}

	\begin{proof}
If $c$ is coherent, then $c^*$ is countably coherent on a point-finite opinion set. Further, $c^*$ is somewhere finitely inaccurate relative to $\mathscr{B}$, as $c$ is somewhere finitely inaccurate by assumption. Thus by Proposition \ref{pointfiniteresult}, $c^*$ is not weakly dominated relative to $\mathscr{B}$. By W-stability, $c$ is not weakly dominated relative to $\mathscr{B}$. Clearly if $c$ is not weakly dominated, then $c$ is not strongly dominated. Finally, we show that if $c$ is incoherent then $c$ is strongly dominated. First, if $c$ is not somewhere finitely inaccurate, then any omniscient credence function strongly dominates $c$ since $\mathscr{I}(v_w,w')<\infty$ for every $w,w'\in W$ by point-finiteness. If $c$ is somewhere finitely inaccurate then $\mathscr{I}(c,w)<\infty$ for all $w\in W$ by point-finiteness. Thus Theorem \ref{inftheorem1} establishes that $c$ is strongly dominated relative to $\mathscr{B}$.
\end{proof}

	\begin{remark}
		We can drop the assumption that all coherent credence functions are somewhere finitely inaccurate in Theorem \ref{brierequivalence} if we strengthen W-stable to compact so that $(W,\mathcal{F})=(W^*,\mathcal{F}^*)$. Indeed, compactness alongside point-finiteness implies coherent credence functions on $(W,\mathcal{F})=(W^*,\mathcal{F}^*)$ are somewhere finitely inaccurate: if there were a coherent (and thus countably coherent) credence function infinitely inaccurate at all worlds, then it would be strongly dominated by an omniscient credence function, contradicting Proposition \ref{strongdominancectblresult} below. 
	\end{remark}
	\subsubsection{Partitions}
	As an application of Theorem \ref{weakdominancemainresult}, we establish Theorem \ref{step3result} for countably infinite partitions.\footnote{It has been noted that \citeauthor{definetti}'s (\citeyear{definetti}) original proof of Theorem \ref{step3result} assuming the Brier score extends to countably infinite opinion sets. However, the only proof we have seen is a sketch of the necessity of coherence for countably infinite partitions by \citeauthor{Joyce1998} (\citeyear[Footnote~6]{Joyce1998}).} In parts of the existing literature (e.g., in \citealt{Joyce}), credence functions are assumed to be defined on a (finite) partition of $W$ to begin with, and so such a result might be especially relevant to extending the accuracy dominance argument for probabilism to countably infinite opinion sets.
	
	\begin{restatable}{lemma}{partitioniswstable}\label{partitioniswstable}
		A countably infinite partition is W-stable relative to any generalized legitimate inaccuracy measure.
	\end{restatable}
		\begin{proof}
	    Let $\mathcal{F}=\{p_1, p_2,\ldots\}$ be a partition.	Assume a coherent credence function $c$ on $\mathcal{F}$ is weakly dominated by some credence function $d$. We can assume $d$ is coherent by Theorem \ref{inftheorem1}, and so $\sum_{m=1}^\infty d_m\leq 1$. We show that $c^*$ is weakly dominated by $d^*$, thereby establishing that a partition is W-stable.
		
        First, $\mathscr{I}(c^*,w)=\mathscr{I}(c,w)$ and $\mathscr{I}(d^*,w)=\mathscr{I}(d,w)$ for all $w\in W$ by Lemma \ref{connectinglemma}. Thus by assumption of weak dominance,
		\[\mathscr{I}(c^*,w)\geq \mathscr{I}(d^*,w)\mbox{ for all } w\in W \]
	    with a strict inequality for some $w\in W$. We therefore need to only check what happens for $w\in W^*\setminus W$. The compactification of a partition consists in adding one point $w^*$ which is in the complement of all $p^*\in \mathcal{F}^*$; so $W^*= W\cup w^*$. If $\mathscr{I}(c^*,w^*)=\infty$, then clearly $d^*$ weakly dominates $c^*$. So assume $\mathscr{I}(c^*,w^*)<\infty$. Since $\mathscr{I}(d^*,w^*)$ and $\mathscr{I}(d^*,w)$ differ by a single term for any $w\in W$ and $\mathscr{I}(d^*,w)<\infty$ for some $w\in W$ by assumption of weak dominance, we have $\mathscr{I}(d^*,w^*)<\infty$. Now, consider 
	   
    	\[\mathscr{I}(c^*,w^*)-\mathscr{I}(d^*,w^*)=\sum_{m=1}^\infty a_m( \varphi(d_m)-\varphi'(d_m)d_m)-\sum_{m=1}^\infty a_m(\varphi(c_m)-\varphi'(c_m)c_m),\]
	    which we claim is greater than or equal to $0$. Indeed, assume toward a contradiction that \[\sum_{m=1}^\infty a_m(\varphi(d_m)-\varphi'(d_m)d_m )< \sum_{m=1}^\infty a_m( \varphi(c_m)-\varphi'(c_m)c_m).\]
	    Then since $d_n\to 0$ as $\sum_{n=1}^\infty d_n\leq 1$, $c_n\to 0$ as $\sum_{n=1}^\infty c_n\leq 1$, and $\varphi'(0)=\lim_{x\to 0}\varphi'(x)=0$ (recall Remark \ref{reductiononphi}), we have that $\varphi'(d_n)-\varphi'(c_n)\to 0$; and so we can find a $K$ such that 
		\[|\varphi'(d_n)-\varphi'(c_n)|<|\sum_{m=1}^\infty a_m(\varphi(d_m)-\varphi'(d_m)d_m)-\sum_{m=1}^\infty a_m(\varphi(c_m)-\varphi'(c_m)c_m)|\]
		for $n\geq K$. Thus for any $n\geq K$ and any $w_n\in p_n$,
		\[\mathscr{I}(c,{w_n})-\mathscr{I}(d,w_n)=\sum_{m=1}^\infty a_m(\varphi(d_m)-\varphi'(d_m)d_m)-\sum_{m=1}^\infty a_m(\varphi(c_m)-\varphi'(c_m)c_m)+\varphi'(d_n)-\varphi'(c_n)<0,\] 	 
	    contradicting that $d$ weakly dominates $c$. So indeed, $d^*$ weakly dominates $c^*$.
		\end{proof}

	\begin{restatable}{theorem}{partitionsresult}\label{partitionsresult}
		Let $(W,\mathcal{F})$ be a countably infinite partition and $\mathscr{I}$ a generalized legitimate inaccuracy measure. Then the following are equivalent:
		\begin{enumerate}
			\item $c$ is coherent;
			\item $c$ is not weakly dominated;
			\item $c$ is not strongly dominated. 
		\end{enumerate}
	\end{restatable}
		\begin{proof}
		The result follows from Corollary \ref{corollarytoweak}, Lemma \ref{partitioniswstable}, and the fact that $\mathscr{I}(c^*,\cdot)$ is bounded on $W^*$ for each coherent credence function $c$. To see the latter, note that since $c$ is coherent it follows that $\sum_{i=1}^\infty c_i=\sum_{i=1}^\infty c^*_i\leq 1$. For $w\in W$ such that $w\in p_i$, recalling that $\varphi(0)=0$ (Remark \ref{reductiononphi}),
		\[\mathscr{I}(c^*,w)=a_i\mathfrak{d}(1,c^*_i)+\sum_{j\neq i}a_j\mathfrak{d}(0,c^*_j)=a_i\mathfrak{d}(1,c^*_i)+\sum_{j\neq i}a_j(c^*_j\varphi'(c^*_j)-\varphi(c^*_j))\leq C+D\sum_{j}c^*_j\leq C+D\]
		for some constants $C,D$ independent of $c^*$. Similarly, as seen in the proof of Lemma \ref{partitioniswstable}, $W^*\setminus W=\{w^*\}$ where 
		\[\mathscr{I}(c^*,w^*)=\sum_{j=1}^\infty a_j \mathfrak{d}(0,c_j^*)=\sum_{j=1}^\infty a_j(c_j^*\varphi'(c_j^*)-\varphi(c_j^*))\leq C\]
		for some constant $C$ independent of $c^*$ or $w$. It follows that i) all coherent credence functions have finite expected inaccuracy and ii) $\mathscr{I}(c,w)<\infty$ for $c\in \mathcal{C}$ and $w\in W$. Thus Lemma \ref{partitioniswstable} and Corollary \ref{corollarytoweak} establish the result.
	\end{proof}

	\subsection{S-Stable Opinion Spaces}
	In this section, we establish the equivalence between coherence and avoiding strong dominance for certain countably infinite opinion spaces (Theorem \ref{strongdominanceequivalence}). The conditions are in terms of the analogous stability condition---\textit{S-stability} (Definition \ref{sstability})---but a different finiteness assumption, and the proof strategy is the same as for Theorem \ref{weakdominancemainresult}.
	
	We begin by establishing that on compact countably infinite opinion spaces, coherent and thus countably coherent credence functions (recall Theorem \ref{borkerresult}) are not strongly dominated.
	\begin{restatable}{proposition}{strongdominancectblresult}\label{strongdominancectblresult}
		Let $(W,\mathcal{F})$ be a compact opinion space and $\mathscr{I}$ a generalized legitimate inaccuracy measure. If $c$ is coherent (and thus countably coherent), then $c$ is not strongly dominated relative to $\mathscr{I}$.
	\end{restatable}
		\begin{proof}
		Let $\mathscr{I}_{n}(c',w):=\sum_{i=1}^na_i\mathfrak{d}(v_w(p_i),c'(p_i))$ for each $n\in \mathbb{N}$, $w\in W$, and  credence function $c'$ on $\mathcal{F}$. Consider a credence function $d\neq c$. Define
		\[T^n=\{(v_w(p_1),\ldots,v_w(p_n)): \mathscr{I}_k(c,w)<\mathscr{I}_k(d,w)\mbox{ for some } k\geq n, w\in W\}\]
		and $T=\{e\}\cup{\bigcup_{n=1}^\infty} T^n$, where $e$ is the empty sequence. For each $s,t\in T$, we set $s<t$ if and only if $s$ is an initial sequence of $t$, and we set the height of $t\in T$ to be the length of the tuple. Then $T$ is a binary tree. 
		
		We claim $T$ is infinite. Fix $n\in \mathbb{N}$. Then there is a $t\in T$ with height $n$ if and only if $T^n\neq \varnothing$ if and only if $\mathscr{I}_k(c,w)<\mathscr{I}_k(d,w)$ for some $k\geq n$ and $w\in W$. Let $k$ be the maximum of $n$ and the smallest $i$ such that $c(p_i)\neq d(p_i)$. Then since $c$ restricted to any subset of $\mathcal{F}$ is coherent, by Theorem \ref{step3result}, $\mathscr{I}_k(c,w')<\mathscr{I}_k(d,w')$ for some $w'\in W$ and so $(v_{w'}(p_1),\ldots,v_{w'}(p_n))\in T^n$.
		
		By Konig's lemma (see, e.g., \citealt{konig}, Sec.~12.3), there exists an infinite branch \[\mathcal{B}=\bigcup_{n=1}^\infty\{ (v_{w_n}(p_1),\ldots,v_{w_n}(p_n))\}\] through $T$, where \[(v_{w_n}(p_1),\ldots,v_{w_n}(p_n))<(v_{w_m}(p_1),\ldots,v_{w_m}(p_m))\]
		whenever $n<m$. For each $i$, let $p_i^*=p_i$ if $v_{w_i}(p_i)=1$ and $p_i^*=p_i^c$ if $v_{w_i}(p_i)=0$. Then $w_n\in \bigcap_{i=1}^np_i^*$ since $v_{w_i}(p_i)=1$ if and only if $v_{w_n}(p_i)=1$ for $i<n$ as $(v_{w_i}(p_1),\ldots, v_{w_i}(p_i))<(v_{w_n}(p_1),\ldots,v_{w_n}(p_n))$. Thus $\bigcap_{i=1}^np_i^*\neq \varnothing$ for each $n$ and so by compactness there is some $w\in \bigcap_{i=1}^\infty p_i^*$. Then \[(v_w(p_1),\ldots ,v_w(p_n))=(v_{w_n}(p_1),\ldots,v_{w_n}(p_n))\in T^n\]
		for each $n \in \mathbb{N}$. By the definition of $T^n$, for each $n\in \mathbb{N}$ we have
		\[\mathscr{I}_{k_n}(c,w)<\mathscr{I}_{k_n}(d,w)\]
		for some $k_n\geq n$. Sending $n$ to infinity, $\mathscr{I}(c,w)\leq \mathscr{I}(d,w)$ and thus $d$ does not strongly dominate~$c$.  
	\end{proof}

    We now introduce S-stability and the main theorem of this section.
	\begin{definition}\label{sstability}
		Let $(W,\mathcal{F})$ be \textit{S-stable relative to} $\mathscr{I}$ if for any coherent credence function $c$ on $(W,\mathcal{F})$, if $c$ is strongly dominated relative to $\mathscr{I}$, then $c^*$ on $(W^*,\mathcal{F}^*)$ is strongly dominated relative to $\mathscr{I}$.
	\end{definition}
	\begin{theorem}\label{strongdominanceequivalence}
		Let $\mathscr{I}$ be a generalized legitimate inaccuracy measure and $(W,\mathcal{F})$ an S-stable opinion space relative to  $\mathscr{I}$. Assume that $\mathscr{I}(c,w)<\infty$ for each coherent credence function $c$ and $w\in W$. Then the following are equivalent:
		\begin{enumerate}
			\item $c$ is coherent;
			\item $c$ is not strongly dominated.
		\end{enumerate}
	\end{theorem}
	\begin{proof}
	Assume $c$ is coherent. $c^*$ defined on the compact opinion space $(W^*,\mathcal{F}^*)$ is countably coherent by Lemma \ref{connectinglemma}. So by Proposition \ref{strongdominancectblresult}, $c^*$ is not strongly dominated. But since $\mathcal{F}$ is S-stable this implies that $c$ is not strongly dominated relative to $\mathscr{I}$. The other direction follows from Theorem \ref{inftheorem1}.
	\end{proof}
		\begin{remark}
		It is trivial to see that S-stability is necessary for the equivalence of coherence and avoiding strong dominance. It is open how much the assumption that coherent credence functions satisfy $\mathscr{I}(c,w)<\infty$ for all $w$ can be weakened.
	\end{remark}
	\begin{remark}
	\cite{1schervish2009} take a different approach to dropping the assumption that the opinion set is finite: they apply weak and strong dominance notions to finite subsets of opinion sets of arbitrary cardinality. They also explore connections between the two notions of dominance considered here---weak and strong dominance---and what they call \textit{coherence$_1$}, which amounts to avoiding being susceptible to a finite \textit{Dutch book}.\footnote{Thanks to Teddy Seidenfeld for pointing me to this work of \citeauthor{1schervish2009}. An additional point worth noting about their work is that they further generalize the finite results of \cite{predd} by i) allowing a wider variety of inaccuracy measures including those which are merely proper as opposed to strictly proper and ii) by scoring conditional probabilities. A natural direction for future work is to use these relaxations in the finite case to relax assumptions made here. Similarly, \cite{STEEGER201934} considers the property of avoiding strong dominance with respect to the Brier score for every finite subset of opinion sets of arbitrary cardinality (see ``sufficient coherence'' on p.~38 of \citealt{STEEGER201934}).}
	\end{remark}
	\begin{remark}
		Theorem \ref{strongdominanceequivalence} is related to Theorem 1 of \citealt{schervishreference}. However, 1) their assumptions are in some ways weaker and in some ways stronger than those in Theorem \ref{strongdominanceequivalence}\footnote{\citeauthor{schervishreference} require that the \textit{prevision} for the inaccuracy of the credence function be finite and that inaccuracy be pointwise finite, while we only assume the latter. On the other hand, we require the opinion set to be $S$-stable while they do not.} and 2) while \cite{schervishreference} establish that coherence is sufficient for avoiding strong dominance in certain cases, unlike Theorems  \ref{weakdominancemainresult} and \ref{strongdominanceequivalence}, their results do not show that coherence is sufficient for avoiding even weak dominance in certain cases or that incoherence always entails being weakly dominated (and sometimes strongly dominated) by a coherent credence function (see Remark \ref{dominatedbycoherent}).
	\end{remark}
		\begin{figure}[t!]
\scalebox{.99}{
\begin{tabular}{@{}ccc@{}}
\textbf{Opinion Space}                                                                                        & \textbf{Inaccuracy Measure}                                                                                                         & \textbf{Conclusion}                                                                                                \\ \midrule
\multicolumn{1}{c|}{-}                                                                                        & \multicolumn{1}{c|}{-}                                                                                                              & \begin{tabular}[c]{@{}c@{}}incoherent $\implies$ weakly dominated\\ (by coherent credence function)\end{tabular}   \\ \midrule
\multicolumn{1}{c|}{-}                                                                                        & \multicolumn{1}{c|}{$\mathscr{I}(c,w)<\infty$ for all $w$}                                                                          & \begin{tabular}[c]{@{}c@{}}incoherent $\implies$ strongly dominated\\ (by coherent credence function)\end{tabular} \\ \midrule
\multicolumn{1}{c|}{W-stable}                                                                                 & \multicolumn{1}{c|}{finite expected inaccuracy}                                                                                     & coherent $\iff$ not weakly dominated                                                                               \\ \midrule
\multicolumn{1}{c|}{W-stable}                                                                                 & \multicolumn{1}{c|}{\begin{tabular}[c]{@{}c@{}}finite expected inaccuracy\\ +\\ $\mathscr{I}(c,w)<\infty$ for all $w$\end{tabular}} & \begin{tabular}[c]{@{}c@{}}coherent $\iff$ not weakly dominated \\ $\iff $ not strongly dominated\end{tabular}     \\ \midrule
\multicolumn{1}{c|}{\begin{tabular}[c]{@{}c@{}}W-stable\\ +\\ $(W^*,\mathcal{F}^*)$ point-finite\end{tabular}} & \multicolumn{1}{c|}{\begin{tabular}[c]{@{}c@{}}$\mathscr{I}=\mathscr{B}$\\ +\\ somewhere finitely inaccurate\end{tabular}}          & \begin{tabular}[c]{@{}c@{}}coherent $\iff$ not  weakly dominated\\ $\iff$ not strongly dominated\end{tabular}      \\ \midrule
\multicolumn{1}{c|}{partition}                                                                                & \multicolumn{1}{c|}{-}                                                                                                              & \begin{tabular}[c]{@{}c@{}}coherent $\iff$ not weakly dominated \\ $\iff$ not strongly dominated\end{tabular}      \\ \midrule
\multicolumn{1}{c|}{S-stable}                                                                                 & \multicolumn{1}{c|}{$\mathcal{I}(c,w)<\infty$ for all $w$}                                                                          & coherent $\iff$ not strongly dominated                                                                             \\ \bottomrule
\end{tabular}}
\caption{This table gives a summary of the results presented in Sections \ref{countablecase1} and \ref{countablecase2}. Each column summarizes a main result by specifying conditions under which a particular conclusion about a credence function $c$ holds. The column title `Opinion Space' specifies conditions on the opinion space on which $c$ is defined, the column titled `Inaccuracy Measure' specifies conditions on the generalized legitimate inaccuracy measure relative to which dominance is defined, and the column titled `Conclusion' specifies what can be deduced regarding the relationship between coherence and dominance for $c$ under the given conditions. `-' in a box means no additional condition is imposed. }
\label{summary}
 \end{figure}
 	For a summary of the results established thus far, see Figure \ref{summary}.
	\subsection{Further Directions}
	While Theorems \ref{weakdominancemainresult} and \ref{strongdominanceequivalence} come close to characterizing the countably infinite opinion spaces on which not being weakly and strongly dominated, respectively, are equivalent to coherence, it is open how far the finiteness assumptions in the theorems can be weakened. This is a natural next line of inquiry. In addition, it would be useful to determine characterizations of W- and S-stability in terms of the inaccuracy measure that make it relatively easy to check whether an opinion set is W- or S-stable. For example, might it be that an opinion space is W- and S-stable relative to a generalized legitimate inaccuracy measure if the generalized legitimate inaccuracy measure only outputs finite scores for credence functions on that opinion space? Also, there are natural ways to generalize the results above to more closely match the finite results: allow different one-dimensional Bregman divergences for different propositions and allow unbounded one-dimensional Bregman divergences. 
	
	Another direction one could go in exploring the sufficiency of coherence for avoiding dominance is as follows: instead of characterizing the countably infinite opinion sets on which Theorem \ref{step3result} goes through, one could characterize the kinds of coherent credence functions for which Theorem \ref{step3result} goes through on any countably infinite opinion set.\footnote{Thanks to Thomas Icard and Milan Moss\'e for suggesting this alternative direction of study.} Doing so might show that while coherence is not enough to avoid dominance in all cases, coherence along with additional plausible constraints is sufficient. In particular, while restricting to finitely supported credence functions is not enough to establish the sufficiency of coherence for avoiding strong dominance (due to Example \ref{problemswithinfinites2}), it is open whether \textit{countable} coherence is equivalent to avoiding weak or strong dominance on the restricted class.

	\section{The Uncountable Case}\label{uncountablecase}
	So far we have been concerned with credences defined on countable opinion sets. We now consider what can be said in favor of probabilism when credences are defined on uncountable opinion sets, though much of what is said will be preliminary. When extending from the finite to the countably infinite setting, we used inaccuracy measures that naturally restrict to legitimate inaccuracy measures in the finite case. When extending from the countable to the the uncountable setting, we will use inaccuracy measures defined as integration against a measure on the uncountable set of propositions. This is a natural generalization of (generalized) legitimate inaccuracy measures, for (generalized) legitimate inaccuracy measures are defined as integration against a particular kind of measure on a countable set of propositions. Indeed, upon inspection, one can see that in the finite and countably infinite setting, a (generalized) legitimate inaccuracy measure defined in terms of weights $\{a_i\}$ and one-dimensional Bregman divergence $\mathfrak{d}$ is given by integrating $\mathfrak{d}(v_w(\cdot),c(\cdot))$ as a function of $\mathcal{F}=\{p_1, p_2,\ldots\}$ against the measure $\mu$ on $\mathcal{F}$ defined by $\mu(A)=\sum_{p_i\in A}a_i$ for $A\in \mathcal{P}(\mathcal{F})$. We generalize this construction for uncountable $\mathcal{F}$ by defining the legitimate inaccuracy measures to be given by integration of a one-dimensional Bregman divergence $\mathfrak{d}(v_w(\cdot),c(\cdot))$ as a function of $\mathcal{F}$ against a measure $\mu$ defined on $\mathcal{F}$.

% 	Similarly, in the uncountable case, we allow for measure theoretically defined inaccuracy measures that naturally restrict to generalized legitimate inaccuracy measures in the countable case. However, for the sake of generality, we allow inaccuracy to be defined by integration against any finite measure.\footnote{Note that since the counting measure over $\mathbb{N}$ is not a finite measure, the result below does not directly establish Theorem \ref{inftheorem1}.} 
	
	Due to the measure theoretic construction of the inaccuracy measures we consider, we restrict our attention to measurable credence functions and equate credence functions that are equal almost everywhere. In some measure spaces, like the weighted counting measure spaces (with all non-zero weights) underlying (generalized) legitimate inaccuracy measures, we lose nothing since every credence function is measurable and only the empty set is measure zero. However, in other cases, these assumptions are substantive. I discuss this issue further after stating the main theorem of this section (Theorem \ref{verygeneral}).
	
    We now formally extend the accuracy framework to the measure theoretic setting.
	\begin{definition}
		Let $\mathcal{F}\subseteq \mathcal{P}(W)$ be an opinion set, $(\mathcal{F},\mathcal{A},\mu)$ a measure space,\footnote{So $\mathcal{A}$ is a $\sigma$-algebra on $\mathcal{F}$ and $\mu$ is a (countably additive) measure on $(\mathcal{F},\mathcal{A})$.} and $c:\mathcal{F}\to \mathbb{R}^+$. If $c$ is $\mathcal{A}$-measurable and $\mu(\{p:c(p)\notin [0,1]\})=0$, we call $c$ a \textit{$\mu$-credence function}. We say that a $\mu$-credence function $c$ is \textit{$\mu$-coherent} if there is a coherent (in the usual sense) credence function $c'$ on $\mathcal{F}$ with $c=c'$ $\mu$-a.e. We say a $\mu$-credence function $c$ is \textit{$\mu$-incoherent} if there is no coherent credence function $c'$ such that $c=c'$ $\mu$-a.e.
	\end{definition}

	\begin{definition}
		Let $\mathcal{F}$ be an opinion set (of arbitrary cardinality) over a set $W$ of worlds. Let $(\mathcal{F},\mathcal{A},\mu)$ be a $\sigma$-finite measure space over the opinion set $\mathcal{F}$.  Let $\mathcal{C}$ be the space of all $\mu$-credence functions. Assume $\mathscr{I}:\mathcal{C}\times W\to [0,\infty]$ is such that, for all $(c,w)\in \mathcal{C}\times W$, we have
		\[\mathscr{I}(c,w)=B_{\varphi,\mu}(v_w,c),\]
	where $B_{\varphi,\mu}$ is a \textit{Bregman distance} relative to $\varphi$\footnote{Again, we assume that the one-dimensional Bregman divergence $\mathfrak{d}$ generated by $\varphi$ is bounded.} and  $(\mathcal{F},\mathcal{A},\mu)$ (see Definition \ref{bregmandistance}). In particular, each $v_w$ is a $\mu$-credence function. Then we call $\mathscr{I}$ an \textit{integral inaccuracy measure on} $(\mathcal{F},\mathcal{A},\mu)$.
	\end{definition}
Here is a dominance result for integral inaccuracy measures. The proof is essentially a measure theoretic version of the proof of Theorem \ref{inftheorem1} and can be found in the Appendix.
	\begin{restatable}{theorem}{verygeneral} \label{verygeneral}Let $\mathscr{I}$ be an integral inaccuracy measure on a finite measure space $(\mathcal{F},\mathcal{A},\mu)$.\footnote{We assume finiteness for technical reasons.} Then for every $\mu$-credence function $c$, if $c$ is $\mu$-incoherent, then there is a $\mu$-coherent $\mu$-credence function  $c'$ that strongly dominates $c$ relative to $\mathscr{I}$.
	\end{restatable}
    
     It is worth noting that this result does \textit{not} show that every incoherent credence function is strongly dominated, since not every incoherent credence function is a $\mu$-incoherent credence function: there can be incoherent credence functions which are $\mu$-a.e. equivalent to coherent credence functions, even ones which are undominated insofar as any coherent credence function is undominated.\footnote{Thanks to an anonymous referee for raising this issue.} This is to be expected since, as discussed above, in moving to a measure theoretic framework, measure zero differences between credence functions will not be detected as far as accuracy is concerned; and thus incoherent credence functions which deviate from coherence by a measure zero set will have the same inaccuracy scores as coherent credence functions. So an accuracy dominance argument in this measure theoretic setting will at most establish that one ought to have a credence function which is coherent off a measure zero set, i.e., is $\mu$-coherent. Theorem \ref{verygeneral} is a first step toward such an argument. We leave for future work considering different ways to extend inacccuracy measures to the uncountable setting that might establish not just $\mu$-coherence but (strict) coherence.
     
    %  Is it a feature or a bug of the measure theoretic framework that one is unable to vindicate (what one might call \textit{strict}) coherence? It is not as leaves open 
     
    %  It depends on the context in which the accuracy of an epistemic state is of concern. If the context is such that there is a representation in the form of a measure $\mu$ of which differences between credence functions are irrelevant (these differences would amount to the measure zero sets)---so the difference between $\mu$-coherence and coherence is irrelevant---then it will be appropriate to define inaccuracy relative to $\mu$. Are there such contexts though? The answer to this is unclear: when credence functions contain so much information---uncountably many credences---perhaps for some applications, ``small'' differences between them need not be of concern. Insofar as there are such applications, there will be cases where a measure theoretic approach to scoring the inaccuracy of uncountable credence functions will be motivated and in which vindicating $\mu$-coherence is all that would be desired.

	Here is an example of how Theorem \ref{verygeneral} can be used to give an accuracy argument in a concrete uncountable setting. Assume we have a coin with unknown bias $\theta\in [0,1]$ and a set of propositions of the form ``$a \leq \theta\leq b$'' for each $a,b\in [0,1]$ with $a\leq b$. Then a credence function on this uncountable opinion set can be represented by a function 
	\[c:X\to [0,1],\]
	where $X=\{(a,b):0\leq a\leq b\leq 1\}\subseteq [0,1]^2$. We put the Lebesgue measure on $X$ to generalize the additive constraint often assumed in the finite case. We let
	\[\mathscr{I}(c,w)=\int_X \mathfrak{d}(v_w(\mathbf{x}),c(\mathbf{x}))\lambda(d\mathbf{x})\]
	for a bounded one-dimensional Bregman divergence $\mathfrak{d}$. Then the assumptions of Theorem \ref{verygeneral} hold, so we get the following dominance result: for any $\lambda$-credence function $c$, if $c$ is $\lambda$-incoherent, then there is a $\lambda$-coherent $\lambda$-credence function that strongly dominates $c$.
	
		\section{Discussion} \label{discussion}
    
    % Assume there are agents that have degrees of belief over infinite opinion sets.\footnote{Perhaps this empirical claim could be rejected, but I leave this objection aside.}
    
    I now briefly consider the difficult question of what normative conclusions to draw from the results of Sections \ref{countablecase1}-\ref{uncountablecase}. 
    \subsection{Extending the Finite Accuracy Dominance Argument for Probablism}
    To begin, it has been shown by \citealt[Theorem 1]{1schervish2009} that a credence function on an opinion set of arbitrary cardinality is coherent if and only if its restriction to any finite subset of the opinion set is not strongly dominated. One might wonder whether this result suffices for an accuracy-based argument for probabilism on an infinite opinion set.\footnote{\cite{DrTruthlove} suggests appealing to a similar notion of ``local accuracy dominance'' when dealing with infinite sets of full beliefs.} Why care in addition about dominance relations when all infinitely many of the agent's credences are scored at once?

    % \footnote{In fact, their Theorem 1 shows that if $c$ is incoherent, then for every restriction $c|_{\mathcal{F}'}$ for $\mathcal{F}'\subseteq \mathcal{F}$ and $|\mathcal{F}'|<\infty$ such that $c_{\mathcal{F}'}$ is incoherent, $c_{\mathcal{F}'}$ is strictly dominated by a coherent credence function $d_{\mathcal{F}'}$ on $\mathcal{F}'$.}
    
    We need to separate the questions of whether there is an accuracy-based argument for probabilism in the infinite setting and whether there is an \textit{extension} to the infinite setting of the accuracy dominance arguments for probabilism already established in the finite setting. While there are differences between the various accuracy dominance arguments for probabilism in the finite setting, they share the key feature of appealing to the following mathematical fact: relative to some legitimate way of scoring inaccuracy of credence functions over a given opinion set, every incoherent credence function is accuracy dominated by a coherent credence function and no coherent credence function is accuracy dominated by any other credence function.\footnote{One notable exception is that unlike \cite{Joyce}, \cite{Joyce1998} does not appeal to the second half of this key fact.} Motivating the appeal to this mathematical fact is a commitment to the epistemic value of having an \textit{overall} accurate epistemic state, where the epistemic state in question is represented by a credence function.
    
    By suggesting that the epistemic value underlying the finite accuracy dominance arguments is that of overall accuracy, I do not mean to suggest that inaccuracy measures ought to be sensitive to what \cite{pettigrew} calls ``irreducibly global features'' of a credence function (see pp.~49-50). To suggest so would likely conflict with allowing an inaccuracy measure to be additive, which I do here. Indeed, Pettigrew  motivates additivity by suggesting that inaccuracy \textit{not} be sensitive to global features since credence functions are not unified doxastic states but simply agglomerations of individual credences.\footnote{I am grateful to an anonymous referee for raising this objection.} The thought here, however, is that in the finite accuracy dominance argument, \textit{all} credences in the agglomeration contribute to the accuracy scores of the representing credence function (perhaps differentially weighted), which are then analyzed for dominance relations. Doing so is motivated by the epistemic ideal of an accuracy undominated epistemic state, where every credal state in the epistemic state is accounted for in scoring. By similarly requiring every credal state in an \textit{infinite} agglomeration to contribute to the accuracy scores of the representing credence function that are then analyzed for dominance relations, I prove results that allow one to extend the accuracy dominance argument while  retaining this underlying motivating ideal.
    
    Now, \citeauthor{1schervish2009}'s Theorem 1 alone does not establish the cited crucial mathematical fact when credence functions are defined over an infinite opinion set, for it suggests no way to score credence functions over infinite opinion sets in the first place. Therefore, an appeal to their result would not be motivated by a commitment to the epistemic value of having an overall accurate epistemic state, a commitment that underwrites the existing accuracy dominance arguments for probabilism. Moreover, while their Theorem 1 establishes that for an incoherent credence function $c$, the restriction of $c$ to any finite subset on which it is incoherent is strictly dominated by a coherent credence function, the dominating credence function may depend on the finite subset.\footnote{In fact, this dependence may hold even in the finite case. See \citealt[Example 10]{1schervish2009}.}  Thus, their result does not provide even an alternative sense in which every incoherent credence function is dominated by a single coherent credence function, which is the kind of result appealed to in the finite accuracy dominance argument.

    Of course, even if a local dominance result like that of \citeauthor{schervishreference}'s cannot be used to extend to the infinite setting what is often referred to as the accuracy dominance argument for probabilism, it could be used to establish a different accuracy-based argument for probabilism in the infinite setting. So we should not assume that, from an accuracy perspective, probabilism in the infinite setting will stand or fall based on the results in this paper---that depends, for instance, on whether it is irrational \textit{in itself} to be in an epistemic state where restrictions of one's representing credence function are accuracy dominated. My point is simply that whatever accuracy-based arguments one can give for probabilism in the infinite setting, it is important to determine whether the influential set of arguments that are often referred to under the single heading of ``the accuracy dominance argument for probabilism'' must be restricted to finite opinion sets. In light of these considerations, to make clear that I will be concerned with comparing the ``total'' inaccuracy scores of credence functions, let us call the inaccuracy (resp.~accuracy) of an agent's entire epistemic state---that is, the inaccuracy of the agent's full credence function---the \textit{total inaccuracy} (resp.~total accuracy) of the agent's epistemic state.

\subsection{The Status of the Accuracy Dominance Argument for Probabilism on Infinite Opinion Sets}
Given the results of the last three sections, what should we conclude about the accuracy dominance argument for probabilism and more generally about the accuracy framework applied to credence functions defined on infinite opinion sets? When an agent's epistemic state contains only finitely many credences, concern for total accuracy leads to a dominance justification for probabilism (using Theorem \ref{step3result}). However, we saw above that when the agent's epistemic state includes even countably infinitely many credences, concern for total accuracy does not so clearly lead to a dominance justification for probabilism since coherence is not sufficient to avoid accuracy dominance in all cases (as in Examples \ref{problemswithinfinities} and \ref{problemswithinfinites2}). I suggest that the main normative challenges at this point are i) to clarify this asymmetry in the relationship between coherence and total accuracy dominance for finite and infinite opinion sets and ii) to determine what this asymmetry means for the accuracy framework more generally.

One response is to deny that real world agents ever have infinitely many credences at a time, and so this entire discussion is mere ideal theory. The empirical claim that real world agents are restricted to finite opinion sets is not obvious, however. For example, if there were a coin in front of me of unknown bias, can I not have credences in each of the propositions ``the bias of the coin is $x$'' for $x\in [0,1]$?\footnote{\citeauthor{pettigrew} (\citeyear[p.~222]{pettigrew}) gives this example to motivate dropping his assumption that the opinion set is finite.} Or can I not have credence $2^{-n}$ in the proposition ``the coin would land heads $n$ times in a row if I flipped it $n$ times in a row'' for each $n\in \mathbb{N}$? More generally, it seems we can have credences parameterized by some infinite set such as the natural numbers. Thus, this kind of objection to considering infinite opinion sets at all seems unlikely to work or, at the very least, requires further defense. 

% Putting aside the uncountable case, another possible response to the purported asymmetry is to normatively justify restricting the legitimate inaccuracy measures for infinite opinion sets to only the generalized legitimate inaccuracy measures that satisfy the conditions of Theorems \ref{weakdominancemainresult} and \ref{strongdominanceequivalence} or Corollary \ref{corollarytoweak} for all countably infinite opinion sets, thereby removing any apparent asymmetry.

Alternatively, one could argue that any asymmetry is unimportant with regard to probabilism: while the results show that some coherent credence functions are ruled out as irrational on the basis of total accuracy dominance, all incoherent credence functions are ruled out on the same basis (in light of Theorems \ref{inftheorem1} and \ref{verygeneral}). Thus, probabilism is justified similarly in both the finite and infinite case: one ought to have at least coherent credences so as to avoid total accuracy dominance. The only difference is that in the infinite case, coherence is not enough to avoid total accuracy dominance. I find this response promising. However, one complication is that while Theorems \ref{inftheorem1} and \ref{verygeneral} together guarantee that incoherent credence functions are dominated by coherent credence functions no matter the infinite opinion set, they do not guarantee dominance by an \textit{undominated} coherent credence function in general.\footnote{It is an important open question whether Theorems \ref{inftheorem1} and \ref{verygeneral} can be strengthened to conclude that the dominating coherent credence function can always be assumed to be undominated.} But, as discussed in Footnote \ref{dominancefootnote}, it has been argued that being accuracy dominated is not irrational in itself, but rather being dominated by a credence function which is itself not dominated is irrational.\footnote{See \citealt[pp.~20-21]{pettigrew} for a counterexample to the stronger principle that dominance alone is irrational.} So this line of response would require proving additional results or appealing to a controversial decision theoretic principle.

A more radical response would be to claim that upon examination, there is simply no reasonable formal measure of total inaccuracy for an agent with infinitely many credences. In particular, no generalized legitimate inaccuracy measure, as I have been calling them, is in fact a legitimate measure of total inaccuracy in the infinite setting. Thus, to justify probablism on any opinion set using accuracy considerations, there is no other option but to appeal to a local dominance result like that of  \citeauthor{1schervish2009}'s. I see three challenges with this response. First, if there is no way to evaluate epistemic states with infinitely many credences for total inaccuracy, i) is the ability to justify norms on epistemic states in terms of total inaccuracy restricted in scope and ii) if so, might this restriction in scope have negative consequences for the more general project of justifying epistemic norms by appealing to total inaccuracy? For example, if concern for total accuracy can only motivate epistemic norms in the finite case so that one must appeal to other epistemic virtues in the infinite case anyway, why care so much about total accuracy in the finite case?

Second, there are infinite opinion sets for which the natural extensions of the inaccuracy measures used in the finite case seem to behave as they do in the finite case with respect to accuracy dominance, e.g., for countably infinite partitions. Do we reject the legitimacy of generalized legitimate inaccuracy measures even in these cases where the measures behave as we would like? If so, why? And if not, then work must be done to spell out precisely for which infinite opinion sets there is no reasonable way to measure total inaccuracy. Presumably, Theorems \ref{weakdominancemainresult} and \ref{strongdominanceequivalence} would be helpful toward this end.

Third, if one is to deny the legitimacy of what I have called generalized legitimate inaccuracy measures, then one needs to explain what exactly is wrong with them, given their structural and axiomatic similarity to the inaccuracy measures often used to measure the total accuracy of credence functions on finite opinion sets.\footnote{For example, generalized legitimate inaccuracy measures satisfy all of \citeauthor{pettigrew}'s (\citeyear{pettigrew}, p.~65) axioms on a legitimate measure of inaccuracy except his decomposition axiom because it is not clear how to even define a key part of the decomposition axiom---what he calls the ``well calibrated counterpart'' of a credence function---in the infinite setting. See also Footnote \ref{axiom}. Clearly, generalized legitimate inaccuracy measures do not satisfy \citeauthor{Joyce}'s (\citeyear{Joyce}) ``coherent admissibility'' in general. However, generalized legitimate inaccuracy measures do satisfy coherent admissibility when restricted to countably infinite partitions (Theorem \ref{partitionsresult}), and \citeauthor{Joyce} restricts to credence functions on partitions in his argument in the finite case. Further analysis of the axiomatic properties of generalized legitimate inaccuracy measures is left for future work.} While one might think that a simple response is that no credence function should be infinitely inaccurate according to a reasonable inaccuracy measure, this response is not so obviously compelling. First, as \cite{pettigrew} points out, it is hard to assess the plausibility of this claim since whether inaccuracy can be infinite or not ``is not something which our concept of accuracy contains much information about'' (p.~37). Another issue with this response, also raised by \citeauthor{pettigrew}, is that the logarithmic inaccuracy measure---often seen as a reasonable way to measure inaccuracy in the finite case (see, e.g., \citealt{logarithm})---can output an infinite inaccuracy score even for finite opinion sets. So if one rejects generalized legitimate inaccuracy measures in virtue of their outputting infinite values, this will have consequences for what the legitimate inaccuracy measures are in the finite case. Lastly, insisting on only inaccuracy measures which output a finite inaccuracy score would justify focusing on  a subclass of the generalized legitimate inaccuracy measures, rather than justify rejecting all of them.\footnote{For example, if the weights defining a generalized legitimate inaccuracy measure are summable, then inaccuracy is always finite.} In fact, it is an interesting open question whether this restriction to the finite generalized legitimate inaccuracy measures picks out a set of generalized legitimate inaccuracy measures for which the finite dominance result goes through to all countably infinite opinion sets.\footnote{\citeauthor{walsh}'s result discussed above (see Theorem \ref{walshtheorem}) may be some evidence in the affirmative.}

To conclude, I do not claim to have provided a full analysis of the options available for responding to the results in this paper, but I hope to have shown that further philosophical work is in order if we are to understand their implications for probabilism and accuracy-based justifications of epistemic norms more generally.

	\section{Conclusion}
	As discussed in the previous section, there is plenty of normative work to be done using the results established above. In light of the failure of coherence being sufficient to avoid strong dominance on certain countably infinite opinion sets, the most pressing question seems to be: is there an accuracy dominance argument for probabilism on at least all countable opinion sets? If not, what does this mean for the accuracy project as a whole? Can we give some sort of privileged status to certain kinds of opinion sets or inaccuracy measures for which coherence is equivalent to not being dominated, e.g., partitions? What is the normative status of the stronger condition of countable coherence?  Further, while the measure theoretic framework introduced in Section \ref{uncountablecase} to score inaccuracy of credence functions over opinion sets of arbitrary cardinality seems like a natural extension of the finite and countably infinite frameworks, is it well motivated that inaccuracy does not track the behavior of a credence function on measure zero sets? The hope with this paper is to start a conversation about these questions by first establishing relevant mathematical results.
	\appendix
	\section{Appendix}
			We review the necessary background before proving Theorems \ref{inftheorem1} and \ref{verygeneral}. 
    \subsection{Generalized Projections}\label{generalizedprojections}
	\citeauthor {csiszar} (\citeyear{csiszar}) showed that what he calls \textit{generalized projections} onto convex sets with respect to Bregman distances exist under very general conditions. We review his relevant results here (but assume knowledge of basic measure theory).  
	\begin{definition}\label{bregmandistance}
		Fix a $\sigma$-finite measure space $(X,\mathcal{X},\mu)$. The \textit{Bregman distance} of non-negative ($\mathcal{X}$-measurable) functions $s$ and $t$ is defined by 
		\[B_{\varphi,\mu}(s,t)=\int \mathfrak{d}(s(x),t(x))\mu(dx)\in [0,\infty]\] 
		where $\mathfrak{d}(s(x),t(x))=\varphi(s(x))-\varphi(t(x))-\varphi'(t(x))(s(x)-t(x))$ for some strictly convex, differentiable function $\varphi$ on $(0,\infty)$.\footnote{For $B_{\varphi,\mu}$ to be a distance measure, we do not need to assume that $\varphi(1)=\varphi'(1)=0$ by the remark following (1.9) in \citealt{csiszar}.}  Note that $B_{\varphi,\mu}(s,t)=0$ iff $s=t$ $\mu$-a.e. See \citealt[p.~165]{csiszar} for details. 
	\end{definition}
	\begin{remark}\label{divergencetodistance}
		Notice that a generalized quasi-additive Bregman divergence $\mathfrak{D}$ with weights $\{a_i\}_{i=1}^\infty$ whose generating one-dimensional Bregman divergence $\mathfrak{d}$ is given in terms of $\varphi$ has a corresponding Bregman distance $B_{\bar{\varphi},\mu}$ with
		\begin{enumerate}
			\item the measure space being $(\mathbb{N},\mathcal{P}(\mathbb{N}),\mu)$, where $\mu(A)=\sum_{i\in A}a_i$ for each $A\in \mathcal{P}(\mathbb{N})$, and
			\item $\bar{\varphi}$ on $(0,\infty)$ being a strictly convex, differentiable extension of $\varphi$ on $[0,1]$.\footnote{Using that $\varphi'$ exists and is finite at $x=1$ as we assumed $\mathfrak{d}$ is bounded, we extend $\varphi$ as follows: for $x\in [1,\infty)$, let $\bar{\varphi}(x)=q(x)=x^2+bx+c$, where $b$ and $c$ are chosen so $\varphi(1)=q(1)$ and $\varphi'(1)=q'(1)$. Then using the fact that $\bar{\varphi}$ is differentiable at $1$ by construction and a function is strictly convex if and only if its derivative is strictly increasing, it is easy to see that $\bar{\varphi}$ is differentiable and strictly convex on $(0,\infty)$.}
		\end{enumerate}
		Thus non-negative $(\mathscr{P}(\mathbb{N})$-measurable) functions are elements of $\mathbb{R}^{+^\infty}$. Note, importantly, that the corresponding generalized legitimate inaccuracy measure $\mathscr{I}$ determined by $\mathfrak{D}$ is also given by the corresponding Bregman distance. That is,
		\[\mathscr{I}(c,w)=B_{\bar{\varphi},\mu}(v_w,c).\]
	\end{remark}
	
	To simplify notation, let $B$ denote $B_{\bar{\varphi},\mu}$ a Bregman distance. Let $S$ be the set of non-negative measurable functions on $(X,\mathcal{X},\mu)$. For any $E\subseteq S$ and $t\in S$, we write 
	\[B(E,t)=\inf_{s\in E}B(s,t).\]
	If there exists $s^*\in E$ with $B(s^*,t)=B(E,t)$, then $s^*$ is unique and is called the \textit{B-projection of $t$ onto $E$} (see \citealt[Lemma 2]{csiszar}). As \citeauthor{csiszar} notes, these projections may not exist. However, a weaker kind of projection exists in a large number of cases. To describe them, we need to introduce a kind of convergence called \textit{loose in $\mu$-measure convergence}.
	\begin{definition}
		We say a sequence $\{s_n\}$ of elements from $S$ converges \textit{loosely in $\mu$-measure} to $t$, denoted by $s_n\rightsquigarrow_\mu t$, if for every $A\in \mathcal{X}$ with $\mu(A)<\infty$, we have 
		\[\lim_{n\to \infty} \mu(A\cap \{p:|s_n(p)-t(p)|>\epsilon\})=0\mbox{ for all }\epsilon>0.\]
	\end{definition}
	\begin{definition} \hfill
		\begin{enumerate}
			\item[i.] Given $E\subseteq S$ and $t\in S$, we say that a sequence $\{s_n\}$ of elements from $E$ is a \textit{$B$-minimizing} sequence if $B(s_n,t)\rightarrow B(E,t)$. 
			\item[ii.] If there is an $s^*\in S$ such that every $B$-minimizing sequence converges to $s^*$ loosely in $\mu$-measure, then we call $s^*$ the \textit{generalized $B$-projection of $t$ onto $E$}.
		\end{enumerate}
	\end{definition}
	The result that is integral to proving Theorem \ref{inftheorem1} is the following (see \citeauthor{csiszar}'s Theorem 1, Lemma 2, and Corollary of Theorem 1).
	\begin{theorem}[\citealt{csiszar}]\label{dominancelemma} Let $E$ be a convex subset of $S$ and $t\in S$. If $B(E,t)$ is finite, then there exists $s^*\in S$ such that 
		\[B(s,t)\geq B(E,t)+B(s,s^*)\mbox{ for every }s\in E\]
		and $B(E,t)\geq B(s^*,t)$. It follows that the generalized $B$-projection of $t$ onto $E$ exists and equals~$s^*$.
	\end{theorem}
	
	\subsection{Extending Partial Measures}
	We also use an extension result of \cite{tarski} in the proof of Theorem \ref{inftheorem1}. Following \citeauthor{tarski}, we introduce \textit{partial measures} and recall that they can be extended to finitely additive probability functions. Recall the definition of a finitely additive probability function in Definition \ref{credence} (though we drop the assumption that $\mathcal{F}$ is finite).
	\begin{remark}\label{simpleremark} It is a simple corollary of the definition of a finitely additive probability function $c$ over an algebra $\mathcal{F}$ that for any $p,p'\in \mathcal{F}$: if $p\subseteq p'$, then $c(p)\leq c(p')$.
	\end{remark}
	Here is another useful fact about finitely additive probability functions. 
	
	\begin{proposition}\label{measureproperty} If $c$ is a finitely additive probability function on an algebra $\mathcal{F}$ and $a_0,\ldots,a_{m-1}\in \mathcal{F}$, then 
		\begin{equation}
		\sum_{k=0}^{m-1}c(a_k)=\sum_{k=0}^{m-1}c(\bigcup_{p\in S^{m,k}}\bigcap_{i\leq k}a_{p_i})
		\end{equation}
		where $S^{m,k}$ is the set of all sequences $p=(p_0,\ldots,p_k)$ with $0\leq p_0<\ldots<p_k<m$.
	\end{proposition}
	To introduce the notion of a partial measure, we need the following definition.
	\begin{definition}Let $\varphi_0,\ldots,\varphi_{m-1}$ and $\psi_0,\ldots,\psi_{n-1}$ be elements of $\mathcal{F}$. Then we write 
		\[(\varphi_0,\ldots,\varphi_{m-1})\subseteq(\psi_0,\ldots,\psi_{n-1})\]
		to mean
		\begin{equation}
		\bigcup_{p\in S^{m,k}}\bigcap_{i\leq k}\varphi_{p_i}\subseteq \bigcup_{p\in S^{n,k}}\bigcap_{i\leq k}\psi_{p_i} \mbox{ for every }k<m
		\end{equation}
		where $S^{r,k}$  ($r=m,n$) is as in Proposition \ref{measureproperty}.\footnote{Note that if $m>n$, this condition implies $\bigcup_{p\in S^{m,k}}\bigcap_{i\leq k}\varphi_{p_i}=\bigcup_{p\in S^{n,k}}\bigcap_{i\leq k}\psi_{p_i}=\varnothing$ for $k\geq n$.}
	\end{definition}
	
	\begin{definition}\label{partialmeasure} A function $c$, defined on a subset $S$ of an algebra $\mathcal{F}$ over $W$, that maps to $\mathbb{R}$ is called a \textit{partial measure} if it satisfies the following properties:
		\begin{enumerate}
			\item \label{partialmeasure1} $c(x)\geq 0$ for $x\in S$;
			\item\label{partialmeasure2} If $\varphi_0,\ldots,\varphi_{m-1},\psi_0,\ldots,\psi_{n-1}\in S$ and 
			\[(\varphi_0,\ldots,\varphi_{m-1})\subseteq (\psi_0,\ldots,\psi_{n-1}),\]
			then 
			\[\sum_{k=0}^{m-1}c(\varphi_{k})\leq \sum_{k=0}^{n-1}c(\psi_{k});\]
			\item\label{partialmeasure3} $W\in S$ and $c(W)=1$.
		\end{enumerate}
	\end{definition}
	The following result is the point of introducing the above definitions. 
	\begin{theorem}[\citealt{tarski}]\label{extendingpartial} Let $c$ be a partial measure on a subset $\mathcal{F}$ of an algebra $\mathcal{A}$. Then there is a finitely additive probability function $c^*$ on $\mathcal{A}$ that extends $c$.
	\end{theorem}
	\subsection{Proof of Theorem \ref{inftheorem1}}
	We now establish the necessity of coherence to avoid dominance in the countably infinite case.
	
	\inftheorem*
	\begin{proof}
		Let $\mathscr{I}$ be a generalized legitimate inaccuracy measure and thus defined by a Bregman distance $B_{\bar{\varphi},\mu}$ (see Remark \ref{divergencetodistance}). We write $B$ for $B_{\bar{\varphi},\mu}$. Let $S$ be the set of non-negative functions on $\mathcal{F}$. Let $E\subseteq S$ be the set of coherent credence functions on $\mathcal{F}$. Then clearly $E$ is convex. 
		
		Let $c$ be an incoherent credence function.
		
		\textbf{Case 1}: $\mathscr{I}(c,w)=\infty$ for all $w\in W$. Then since $\mathscr{I}(v_{w},w)=0$ for all $w\in W$, any omniscient credence function weakly dominates $c$. 
		
		\textbf{Case 2}: $\mathscr{I}(c,w')<\infty$ for some $w'\in W$. We show that there is a coherent credence function $\pi_c$ such that \[\mathscr{I}(c,w)>\mathscr{I}(\pi_c,w) \mbox{ for any } w \mbox{ such that } \mathscr{I}(c,w)<\infty.\] Since $v_{w'}\in E$, we see that
		\[B(E,c)\leq B(v_{w'},c)=\mathscr{I}(c,w')<\infty.\]
		Thus we can apply Theorem \ref{dominancelemma} to get a $\pi_c\in S$ such that
		\begin{equation}\label{inequality}
		B(s,t)\geq B(E,c)+B(s,\pi_c)\mbox{ for every }s\in E.
		\end{equation}
		In particular, (\ref{inequality}) holds when $s$ is the omniscient credence function at world $w$ for any $w\in W$; and so we see that
		\begin{equation}\label{inequality2}
		\mathscr{I}(c,w)\geq B(E,c)+\mathscr{I}(\pi_c,w)
		\end{equation}
		for all $w$, where all numbers in (\ref{inequality2}) are finite whenever $\mathscr{I}(c,w)<\infty$. 
		
		Next we show that $\pi_c$ is in fact coherent. This is due to the  following claim: $E$ is closed under loose convergence in $\mu$-measure where $\mu$ is a weighted counting measure on $\mathcal{P}(\mathbb{N})$ defined with weights $\{a_i\}_{i=1}^\infty$. To see this, let $c_n\in E$ for each $n$ and $c\in S$. Assume $c_n\to c$ loosely in $\mu$-measure. We show $c\in E$, i.e., $c$ is coherent. Note $c$ is coherent on $\mathcal{F}$ if and only if $c':\mathcal{F}\cup \{W\}\to [0,1]$ is coherent on $\mathcal{F}\cup \{W\}$, where $c'=c$ on $\mathcal{F}$ and $c'(W)=1$. Thus it suffices to assume $c$ and $c_n$ for all $n$ are defined on $\mathcal{F}\cup \{W\}$ with $c(W)=c_n(W)=1$ for all $n$.
		
		It is easy to see that loose convergence in a weighted counting measure (where all weights are non-zero) implies pointwise convergence on $\mathcal{F}$, so
		\[c(p)=\lim_{n\to \infty}c_n(p)\in [0,1]\]
		for each $p\in \mathcal{F}\cup \{W\}$. To show $c\in E$, it suffices to show $c$ can be extended to a finitely additive probability function on $\mathcal{P}(W)$. 
		
		We first show $c$ is a partial measure on $\mathcal{F}\cup \{W\}$. Definitions \ref{partialmeasure}.\ref{partialmeasure1} and \ref{partialmeasure}.\ref{partialmeasure3} clearly hold for $c$ so we just need to show Definition \ref{partialmeasure}.\ref{partialmeasure2} holds. Let $\varphi_0,\ldots,\varphi_{m-1},\psi_0,\ldots,\psi_{m'-1}\in \mathcal{F}\cup \{W\}$ and
		\[\bigcup_{p\in S^{m,k}}\bigcap_{i\leq k}\varphi_{p_i}\subseteq \bigcup_{p\in S^{m',k}}\bigcap_{i\leq k}\psi_{p_i}\]
		for every $k<m$. Since the $c_n$ are coherent and thus extend to finitely additive probability functions on algebras containing $\mathcal{F}$, we have by Proposition \ref{measureproperty} and Remark \ref{simpleremark} that
		\[\sum_{k=0}^{m-1}c_n(\varphi_k)=\sum_{k=0}^{m-1}c_n(\bigcup_{p\in S^{m,k}}\bigcap_{i\leq k}\varphi_{p_i})\leq\sum_{k=0}^{m'-1}c_n(\bigcup_{p\in S^{m',k}}\bigcap_{i\leq k}\psi_{p_i})=\sum_{k=0}^{m'-1}c_n(\psi_k)\]
		using that 
		\[\bigcup_{p\in S^{m,k}}\bigcap_{i\leq k}\varphi_{p_i}=\bigcup_{p\in S^{m',k}}\bigcap_{i\leq k}\psi_{p_i}=\varnothing\]
		for $k\geq m'$. Sending $n$ to infinity and using the pointwise convergence of $c_n$ to $c$ on $\mathcal{F}\cup \{W\}$ we obtain that 
		\[\sum_{k=0}^{m-1}c(\varphi_k)\leq \sum_{k=0}^{m'-1}c(\psi_k).\]
		Thus $c$ is a partial measure on $\mathcal{F}\cup \{W\}$. By Theorem \ref{extendingpartial}, it follows that there is a finitely additive probability function $c^*$ on an algebra $\mathcal{F}^*\supseteq \mathcal{F}$ that extends $c$ and so $c\in E$, which concludes the proof that $E$ is closed under loose $\mu$-convergence.
		
		By Theorem \ref{dominancelemma}, $\pi_c$ is the generalized $B$-projection of $c$ onto $E$. Also, since \[B(E,c)=\inf_{s\in E}(s,c)<\infty,\]
		there is a B-minimizing sequence $\{s_n\}\subseteq E$ such that $B(s_n,c)\to B(E,c)$ by the definition of infimum. By the definition of a generalized projection, $s_n\rightsquigarrow_\mu \pi_c$. Since $E$ is closed under loose convergence, it follows that $\pi_c\in E$. Further, by Theorem \ref{dominancelemma}, 
		\[B(E,c)\geq B(\pi_c,c)>0,\]
		since $\pi_c\neq c$ (as $c$ is incoherent) and $B(s,t)=0$ if and only if $s=t$ (as $\mu$ is a weighted counting measure with all non-zero weights). So for every $w$ such that $\mathscr{I}(c,w)<\infty$, we deduce that
		\[\mathscr{I}(c,w)\geq B(E,c)+\mathscr{I}(\pi_c,w)>\mathscr{I}(\pi_c,w).\]
		This proves that $c$ is weakly dominated by $\pi_c$, and $c$ is strongly dominated by $\pi_c$ if $\mathscr{I}(c,w)<\infty$ for all $w\in W$.
	\end{proof}
	\subsection{Proof of Theorem \ref{verygeneral}}	 
		We now establish the necessity of coherence to avoid dominance in the uncountable case.

\verygeneral*
		\begin{proof}
		Let $\mathscr{I}(c,w)=B_{\varphi,\mu}(v_w,c)$. We write $B$ for $B_{\varphi,\mu}$. Let $S$ be the set of non-negative $\mathcal{A}$-measurable functions on $\mathcal{F}$. Let $E\subseteq S$ be the set of $\mu$-coherent $\mu$-credence functions over $\mathcal{F}$. Then $E$ is convex. Let $c$ be a $\mu$-incoherent $\mu$-credence function. Because $\mu$ is finite and $\mathfrak{d}$ is bounded,
		\[B(E,c)<\infty.\]
		Thus we can apply Theorem \ref{dominancelemma} to get a $\pi_c\in S$ such that
		\begin{equation}\label{inequality'}
		B(s,c)\geq B(E,c)+B(s,\pi_c)\mbox{ for every }s\in E.
		\end{equation}
		In particular, (\ref{inequality'}) holds when $s$ is the omniscient credence function at world $w$ for each $w$, so we obtain
		\begin{equation}\label{inequality2'}
		\mathscr{I}(c,w)\geq B(E,c)+\mathscr{I}(\pi_c,w)
		\end{equation}
		for all $w$, where all numbers in (\ref{inequality2'}) are finite. We show that $\pi_c$ is in fact a $\mu$-coherent $\mu$-credence function. It suffices to show that $\pi_c$ is $\mu$-a.e. equal to a coherent credence function on $\mathcal{F}$ (since $\pi_c\in S$, it is $\mathcal{A}$-measurable). To do so, we prove the following claim: $E$ is closed under loose-convergence in $\mu$-measure.
		
		To see this, let $c_n\in E$ for each $n$ and $c\in S$. Assume $c_n\to c$ loosely in $\mu$-measure.
		The first thing to notice is that, since $\mu$ is finite, loose $\mu$-convergence implies $\mu$-a.e. convergence on a subsequence $\{a_n\}_{n=1}^\infty$ of $\{n\}_{n=1}^\infty$,\footnote{It is a standard fact that convergence in measure implies a.e. convergence on a subsequence. Now notice that loose convergence implies convergence in measure when the measure is finite.} so that
		\[c(p)=\lim_{n\to \infty}c_{a_n}(p)\in [0,1]\]
		for each $p\in \mathcal{G}$ with $\mu(\mathcal{G}^c)=0$. Since the $c_{a_n}$ are $\mu$-coherent, we can change each $c_{a_n}$ on a (measurable) measure zero set $\mathcal{X}_n$ to get coherent $\mu$-credence functions $c_{a_n}$. Further, we replace $\mathcal{G}$ with $\mathcal{G}\setminus (\cup_{n=1}^\infty \mathcal{X}_n)$. Assuming these adjustments have been made, we have that $c_{a_n}\to c$ on $\mathcal{G}$ with $\mu(\mathcal{G}^c)=0$, and  each $c_{a_n}$ is coherent. We now show $c\in E$ by showing it is equal to a coherent credence function on $\mathcal{F}$ when restricting to $\mathcal{G}$. 
		
		First, we extend $c$ (resp.~$c_{a_n}$) to $\overline{c}$ (resp.~$\overline{c_{a_n}}$), where $\overline{c}$ (resp.~$\overline{c_{a_n}}$) is a credence function on $\mathcal{G}\cup \{W\}$ such that $c=\overline{c}$ (resp.~$c_{a_n}=\overline{c_{a_n}}$) on $\mathcal{G}$ and $\overline{c}(W)=1$ (resp.~$\overline{c_{a_n}}(W)=1$). Then notice that $c$ (resp.~$c_{a_n}$) is coherent on $\mathcal{G}$ if and only if $\overline{c}$ (resp.~$\overline{c_{a_n}}$) is coherent on $\mathcal{G}\cup \{W\}$. Thus we work with $\overline{c}$ and $\overline{c_{a_n}}$ instead noting that $\overline{c}=\lim_n\overline{c_{a_n}}$ on $\mathcal{G}\cup \{W\}$. To show $\overline{c}\in E$, we first show $\overline{c}$ is a partial measure on $\mathcal{G}\cup \{W\}$.  
		
		Definitions \ref{partialmeasure}.\ref{partialmeasure1} and  \ref{partialmeasure}.\ref{partialmeasure3} clearly hold for $\overline{c}$ so we just need to show that Definition \ref{partialmeasure}.\ref{partialmeasure2} holds. Let $\varphi_0,\ldots,\varphi_{m-1},\psi_0,\ldots,\psi_{m'-1}\in \mathcal{G}\cup \{W\}$ and
		\[\bigcup_{p\in S^{m,k}}\bigcap_{i\leq k}\varphi_{p_i}\subseteq \bigcup_{p\in S^{m',k}}\bigcap_{i\leq k}\psi_{p_i}\]
		for every $k<m$. Since $\overline{c_{a_n}}$ are coherent on $\mathcal{G}\cup \{W\}$ and thus extend to measures on an algebra containing $\mathcal{G}\cup \{W\}$, we have by Corollary \ref{measureproperty} that
		\[\sum_{k=0}^{m-1}\overline{c_{a_n}}(\varphi_k)=\sum_{k=0}^{m-1}\overline{c_{a_n}}(\bigcup_{p\in S^{m,k}}\bigcap_{i\leq k}\varphi_{p_i})\leq\sum_{k=0}^{m'-1}\overline{c_{a_n}}(\bigcup_{p\in S^{m',k}}\bigcap_{i\leq k}\psi_{p_i})=\sum_{k=0}^{m'-1}\overline{c_{a_n}}(\psi_k)\]
		using that
		\[\bigcup_{p\in S^{m,k}}\bigcap_{i\leq k}\varphi_{p_i}=\bigcup_{p\in S^{m',k}}\bigcap_{i\leq k}\psi_{p_i}=\varnothing\]
		for $k\geq m'$. Sending $n$ to infinity and using the pointwise convergence of $\overline{c_{a_n}}$ to $\overline{c}$ on $\mathcal{G}\cup \{W\}$ we conclude that 
		\[\sum_{k=0}^{m-1}\overline{c}(\varphi_k)\leq \sum_{k=0}^{m'-1}\overline{c}(\psi_k).\]
		Thus $\overline{c}$ is a partial measure on $\mathcal{G}\cup \{W\}$. By Theorem \ref{extendingpartial}, it follows that there is a finitely additive probability function $c^*$ on $\mathcal{A}(\mathcal{F})$ such that $c^*=\overline{c}$ on $\mathcal{G}\cup \{W\}$. Thus  $c^*|_\mathcal{F}$ is a coherent credence function on $\mathcal{F}$ and 
		\[c=\bar{c}|_{\mathcal{F}}=c^*|_{\mathcal{F}}\] $\mu$-a.e.~(specifically off $\mathcal{G}^c$). Further, we already assumed $c$ is $\mathcal{A}$-measurable and $\{p:c(p)\in [0,1]\}\subseteq \mathcal{G}$. Thus $c$ is a $\mu$-coherent $\mu$-credence function.
		
		The proof is finished just as in the proof of Theorem \ref{inftheorem1}. By Theorem \ref{dominancelemma}, $\pi_c$ is the generalized projection of $c$ onto $E$. Since \[B(E,c)=\inf_{s\in E}(s,c)<\infty\] there is a B-minimizing sequence $\{s_n\}$ of elements in $E$ such that $B(s_n,c)\to B(E,c)$ by the definition of infimum. By the definition of a generalized projection, $s_n\rightsquigarrow_\mu \pi_c$. Since $E$ is closed under loose convergence, it follows that $\pi_c\in E$. Further, since $c$ is $\mu$-incoherent we know $c\neq \pi_c$ (up to $\mu$-a.e. equivalence) so we see $B(E,c)\geq B(\pi_c,c)>0$ since $B(s,t)=0$ if and only if $s=t$ $\mu$-a.e. Since $\mathscr{I}(c,w)<\infty$ for all $w$, we deduce that
		\[\mathscr{I}(c,w)\geq B(E,c)+\mathscr{I}(\pi_c,w)>\mathscr{I}(\pi_c,w)\]
		for all $w\in W$. This proves that $c$ is strongly dominated by $\pi_c$, and we are done.
	\end{proof}

	\bibliographystyle{plainnat}
	\bibliography{accuracycitations}
\end{document}